\documentclass{article}
\usepackage[utf8]{inputenc}
\usepackage[a4paper,margin=1.5cm,footskip=.5cm]{geometry}

\usepackage{amsmath, amsfonts, amssymb, amsthm} 
\usepackage{biblatex}
\usepackage{mathrsfs}
\usepackage{mathtools}
\usepackage{enumitem} 
    \setlist[enumerate]{label=(\roman*),noitemsep,topsep=3pt} 
    \setlist[itemize]{label={$\cdot$},noitemsep,topsep=3pt} 
\usepackage{verbatim} 
\usepackage{hyperref} 
\usepackage{url}
\usepackage{theoremref} 
\usepackage{bbm}
\usepackage{color}
\usepackage{float}
\definecolor{gray}{rgb}{0.5,0.5,0.5}
\usepackage{subcaption}
\usepackage{adjustbox}
\usepackage{tikz}
    \usetikzlibrary{matrix}
    \usetikzlibrary{backgrounds}
    \usetikzlibrary{positioning}
    
    \newdimen\nodeSize
    \newdimen\nodeDist
    \tikzset{
        position/.style args={#1:#2 from #3}{
            at=(#3.#1), anchor=#1+180, shift=(#1:#2)
        }
    }

\usepackage{listings}

\usepackage{svg}
\usepackage{rotating}
\usepackage{makecell}
\usepackage{tabu}
\usepackage{multirow}

\nocite{*} 

\makeatletter
\newcommand\ackname{Acknowledgements}
\if@titlepage
  \newenvironment{acknowledgements}{%
      \titlepage
      \null\vfil
      \@beginparpenalty\@lowpenalty
      \begin{center}%
        \bfseries \ackname
        \@endparpenalty\@M
      \end{center}}%
     {\par\vfil\null\endtitlepage}
\else
  \newenvironment{acknowledgements}{%
      \if@twocolumn
        \section*{\abstractname}%
      \else
        \small
        \begin{center}%
          {\bfseries \ackname\vspace{-.5em}\vspace{\z@}}%
        \end{center}%
        \quotation
      \fi}
      {\if@twocolumn\else\endquotation\fi}
\fi
\makeatother

\newtheorem{theorem}{Theorem}[section]
\newtheorem{corollary}[theorem]{Corollary}
\newtheorem{proposition}[theorem]{Proposition}
\newtheorem{lemma}[theorem]{Lemma}

\theoremstyle{definition}
\newtheorem{definition}[theorem]{Definition}
\newtheorem{remark}[theorem]{Remark}
\newtheorem{problem}{Problem}

\newcommand{\thistime}{\expandafter\calctimeA\pdfcreationdate\@nil} 
\def\calctimeA#1:#2#3#4#5#6#7#8#9{\calctimeB}
\def\calctimeB#1#2#3#4#5\@nil{#1#2:#3#4}

\newcommand{\nor}{\mathsf{N}} 

\newcommand{\C}{\mathsf{C}}
\newcommand{\D}{\mathsf{D}}
\newcommand{\Q}{\mathsf{Q}}
\newcommand{\SD}{\mathsf{SD}}
\newcommand{\M}{\mathsf{M}}
\newcommand{\V}{\mathsf{V}}

\newcommand{\numberset}[1]{\mathbbm{#1}}
\newcommand{\N}{\numberset{N}}
\newcommand{\Z}{\numberset{Z}}

\newcommand{\Zmod}[1]{\Z/{#1}\Z}

\newcommand{\sym}{\operatorname{\mathsf{Sym}}}

\newcommand{\aut}{\operatorname{\mathsf{Aut}}}

\newcommand{\hol}{\operatorname{\mathsf{Hol}}}
\renewcommand{\ker}{\operatorname{\mathsf{ker}}}

\newcommand{\inv}{\mathsf{inv}}
\newcommand{\id}{\mathsf{id}}

\renewcommand{\max}{\operatorname{\mathsf{max}}}

\newcommand{\+}{\mkern2mu} 
\let\perogni\forall 
\let\esiste\exists
\renewcommand{\forall}{\perogni\+}
\renewcommand{\exists}{\esiste\+}

\newcommand{\g}{\gamma}
\renewcommand{\d}{\delta}
\newcommand{\s}{\sigma}
\renewcommand{\c}{\circ}

\def\ceil#1{\left\lceil{#1}\right\rceil}
\def\floor#1{\left\lfloor{#1}\right\rfloor}

\newcommand{\norm}{\trianglelefteq} 
\newcommand{\mo}{^{-1}} 

\newcommand{\omo}{^{\mathrlap{\,\circ}{-} 1}}

\title{The mutually normalizing regular subgroups of the holomorph of a cyclic group of prime power order}
\author{Filippo Spaggiari}
\date{\today}

\begin{document}

\maketitle

\begin{abstract}
    Let $G=\C_{p^n}$ be a finite cyclic $p-$group, and let $\hol(G)$ denote its holomorph. In this work, we find and characterize the regular subgroups of $\hol(G)$ that mutually normalize each other in the symmetric group $\sym(G)$. We represent such regular subgroups as vertices of a graph, and we connect a pair of them by an edge when they mutually normalize each other. The approach to construct this \emph{local normalizing graph} relies on the theory of \emph{gamma functions}, and the final result will contain all the information about the regular subgroups of $\hol(G)$ in a compact form.
\end{abstract}

\section{Introduction}

Given a pointed set $(G,1)$, we consider the undirected graph whose vertices are regular subgroups of $\sym(G)$, where two of them are joined by an edge if they mutually normalize each other. This \emph{normalizing graph} $\mathcal{G}$ has attracted some interest, among others, because of its connections with the recently developed theory of skew braces (see \cite{caranti2021brace}, \cite{caranti2022skew}, and \cite{childs2019biskew}). For each regular subgroup $N$ of $\sym(G)$, the bijection
    \begin{align*}
        N \to G\\
        n \mapsto 1^{n}
    \end{align*}
induces, by transport, a group structure $(G, \circ_{N},1)$ on $G$. If $\mathcal{C}$ is a clique (that is, a complete subgraph) of $\mathcal{G}$, then for each $N,M \in\mathcal{C}$ we have that $(G, \circ_{N}, \circ_{M})$ is a skew brace.

In general, a clique in the \emph{normalizing graph} is equivalent to a \emph{brace block}, a notion introduced and used by A. Koch in \cite{koch2020abelian} and \cite{koch2021abelian} to produce some non-trivial solutions of the \emph{Yang-Baxter Equation}. In this paper, for all cyclic groups $G = (G, \cdot, 1)$ of prime power order, we determine the subgraph of $\mathcal{G}$ consisting of the regular subgroups that normalize the image $\rho(G)$ of the right regular representation of $G$, that is, they lie in the (permutational) holomorph $\hol(G)$ of $G$. We distinguish the cases where the order of $G$ is a power of $2$ and where it is a power of an odd prime. The case $p=2$ turns out to be more difficult to grasp, requiring several exceptions and a multi-directional approach. On the other hand, if compared with the previous one, the case of $p$ odd appears as a simple and straightforward generalization. Therefore, we will study them separately and in different sections of this work (Sections \ref{sec:p2} and \ref{sec:podd}).
    
To accomplish this subgraph, later called the \emph{local normalizing graph}, and to classify the regular subgroups of the holomorph, we employ the language and method of \emph{gamma functions}, as outlined in \cite{Campedel_2020}, \cite{Caranti_2020}, \cite{Caranti_2018}. Beyond the complete description of the local normalizing graph in the cases mentioned above, the main result of this work is \thref{thm:neo}, which is a characterization of the mutual normalization of regular subgroups in terms of a pair of equations. In the case of cyclic groups, such equations will be easily translated in modular arithmetic, providing a simple way to check whether they are valid or not.

\section{Preliminaries} \label{preliminaries}

In this section we present the notions and the main results we are going to use later, in the development of the local normalizing graphs. We start by recalling the definition of the algebraic environment in which we will conduct the study. Fixed a group $G$, we denote by $\id$ and $\inv$ the identity and the inversion map on $G$ respectively, by $\rho\colon G\to\sym(G), g\mapsto(x\mapsto xg)$ the right regular representation, and by $\iota\colon G\to G, g\mapsto(x\mapsto gxg\mo)$ the conjugation map.

\begin{definition}
    Let $G$ be a group. We define the \textbf{(permutational) holomorph} $\hol(G)$ of $G$ as the subgroup of $\sym(G)$ generated by $\rho(G)$ and $\aut(G)$, that is $$\hol(G) =\langle\rho(G),\aut(G)\rangle.$$
\end{definition}

A powerful characterization is known for the regular subgroups of the holomorph of a given group $G$, by using the notions of \emph{gamma function} and \emph{skew brace}. We recall these concepts and the aforementioned result, from \cite{Campedel_2020}, \cite{Caranti_2020}, and \cite{Caranti_2018}. For the purposes of this paper, we restrict attention to the case of finite groups.

\begin{definition}
    A \textbf{(right) skew brace} on a group $(G,\cdot)$ is a triple $(G,\cdot,\circ)$ where $\circ$ is an operation on $G$ such that $(G,\circ)$ is also a group and the following axiom holds 
    \begin{equation} \label{eqn:skew_braces_axiom}
        (x\cdot y)\circ z = (x\circ z)\cdot z\mo\cdot(y\circ z)\quad\quad\quad\forall x,y,z\in G.
    \end{equation}
    The two groups $(G,\cdot)$ and $(G,\circ)$ are called the \textbf{additive group} and \textbf{circle group} of the skew brace, respectively. A skew brace $(G,\cdot,\circ)$ is called \textbf{bi-skew brace} if $(G,\circ,\cdot)$ is also a skew brace.
\end{definition}

\begin{definition} \thlabel{def:gamma_function}
    Let $(G,\cdot)$ be a group, let $A\le G$ and let $\gamma\colon A\to \aut(G,\cdot)$ be a function. Then $\gamma$ is said to satisfy the \textbf{gamma functional equation} on $A$ if $$\gamma(g^{\gamma(h)}\cdot h)=\gamma(g)\gamma(h)\quad\quad\quad\forall g,h\in A.$$ The function $\gamma$ is said to be a \textbf{relative gamma function} on $A$ if it satisfies the gamma functional equation on $A$ and $A$ is $\gamma(A)$-invariant. If $A=G$, a relative gamma function on $G$ is simply called \textbf{gamma function} on $G$.
\end{definition}

\begin{theorem}[\cite{Caranti_2020}] \thlabel{thm:srg1}
    Let $(G,\cdot)$ be a finite group. The following data are equivalent.
    \begin{enumerate}
        \item A regular subgroup $N\le\hol(G,\cdot)$. \thlabel{thm:srg1_1}
        \item A gamma function $\gamma\colon (G,\cdot)\to\aut(G,\cdot)$. \thlabel{thm:srg1_2}
        \item A group operation $\circ$ on $G$ such that $(G,\cdot,\circ)$ is a skew brace.
    \end{enumerate}
    Moreover, the previous data are related as follows.
    \begin{enumerate}[label=(\alph*)]
        \item Each element of $N$ can be written uniquely in the form $\nu(g)=\gamma(g)\rho(g)$ for some $g\in G$.
        \item For all $g,h\in G$ we have $g\circ h=g^{\gamma(h)}\cdot h$.
        \item For all $g,h\in G$ we have $g^{\nu(h)}=g\circ h$.
        \item For every $g\in G$ we have $\gamma(g)\mo=\gamma(g\omo)$, where $g\omo$ denotes the inverse of $g$ with respect to the circle operation $\circ$.
        \item The map $\gamma\colon(G,\circ)\to\aut(G,\cdot)$ is a group homomorphism.
        \item The map
        \[
        \begin{aligned}
            \nu\colon(G,\circ)&\to N\\
            g&\mapsto\gamma(g)\rho(g)
        \end{aligned}
        \]
        is a group isomorphism.
    \end{enumerate}
\end{theorem}

The property of being regular is preserved by conjugation under any permutation, in particular under any automorphism. Moreover, the gamma function associated with the conjugated regular subgroup under an automorphism can be obtained with a very simple formula. We summarize all these observations in the following proposition. 

\begin{proposition}[\cite{Campedel_2020}] \thlabel{thm:gamma_alpha_lemma}
    Let $G$ be a finite group, let $N\le\hol(G)$ be a regular subgroup with associated gamma function $\gamma$, and let $\alpha\in\aut(G)$. Then
    \begin{enumerate}
    
        \item $N^\alpha$ is a regular subgroup of $\hol(G)$.
        
        \item The gamma function $\gamma^\alpha$ associated with the regular subgroup $N^\alpha$ is given by $$\gamma^\alpha(g)=\gamma(g^{\alpha\mo})^\alpha=\alpha\mo\gamma(g^{\alpha\mo})\alpha\quad\quad\forall g\in G.$$
        
        \item If $H\le G$ is $\gamma(H)$-invariant, then $H^\alpha$ is $\gamma^\alpha(H^\alpha)$-invariant.
        
        \item Denote by $\circ$ and $\circ_\alpha$ respectively the circle operations associated with $\gamma$ and $\gamma^\alpha$. Then the map $$\alpha\mo\colon(G,\circ_\alpha)\to (G,\circ)$$ is an isomorphism of groups. In particular $N^\alpha\cong N$.
        
    \end{enumerate}
\end{proposition}

Since $\aut(G)$ is closed under composition, it is natural to transport the idea of conjugation of subgroups to the gamma functions, interpreting the construction in \thref{thm:gamma_alpha_lemma} as an \emph{action by conjugation} on the gamma functions. We formalize this intuition in the following definition.

\begin{definition} \thlabel{thm:gamma_conj_class}
    Let $G$ be a finite group and let $N\le\hol(G)$ be a regular subgroup with associated gamma function $\gamma\colon G\to\aut(G)$. Consider the set $\Gamma = \left\{\gamma^\alpha\, :\,\alpha\in\aut(G)\right\}.$ Then the map
    \[
    \begin{aligned}
        \Gamma\times\aut(G)&\to\Gamma \\
        (\gamma^\alpha,\beta) &\mapsto (\gamma^\alpha)^\beta=\gamma^{\alpha\beta}
    \end{aligned}
    \]
    defines a transitive action on $\Gamma$, where $\gamma^{\alpha\beta}$ is the gamma function associated with the regular subgroup $N^{\alpha\beta}$, as in \thref{thm:gamma_alpha_lemma}. We refer to this action as the \textbf{conjugation by $\alpha\in\aut(G)$}.
\end{definition}

Due to reasons explained in the following, we need regularity of such action by conjugation, instead of transitivity only. 

\begin{lemma} \thlabel{thm:gamma_conj_class_regular}
    Let $G$ be a finite group such that $\aut(G)$ is abelian and let $N\le\hol(G)$ be a regular subgroup with associated gamma function $\gamma\colon G\to\aut(G)$. Denote by $\nor_{\sym(G)}(N)$ the normalizer of $N$ in $\sym(G)$ and consider the set $\Gamma = \left\{\gamma^\alpha\, :\,\alpha\in\aut(G)\right\}$ and the subgroup $M=\nor_{\sym(G)}(N)\cap\aut(G)\le\aut(G)$. Then the factor group $$E=\frac{\aut(G)}{M}$$ acts regularly on $\Gamma$, in the sense that the map
    \[
    \begin{aligned}
        \Gamma\times E &\to\Gamma \\
        (\gamma^\alpha,M\beta) &\mapsto (\gamma^\alpha)^\beta=\gamma^{\alpha\beta}
    \end{aligned}
    \]
    defines a regular action on $\Gamma$.
\end{lemma}

When no confusion can arise, we will still refer to this action as \emph{conjugation}. We will see that the local normalizing graphs are composed of subgraphs whose orders are strictly related to the size of the conjugacy classes of this action. Therefore, it will be useful the following straightforward result.

\begin{corollary} \thlabel{thm:gamma_counts}
    In the notation of \thref{thm:gamma_conj_class_regular}, $$|\Gamma|=|E|=\frac{|\!\aut(G)|}{|\nor_{\sym(G)}(N)\cap\aut(G)|}$$ is the cardinality of the conjugacy class under automorphisms of a given regular subgroup $N\le\hol(G)$.
\end{corollary}

\begin{definition}
    Let $G$ be a finite group. Given $g\in G$ and $\alpha\in\aut(G)$, we denote by $[g,\alpha] = g\mo g^\alpha$. If $A\le G$ and $\gamma\colon A\to\aut(G)$ is a function, we denote by $$[A,\gamma(A)] =\Big\{[x,\gamma(y)]\, :\, x,y\in A\Big\}.$$
\end{definition}

It is not true, \emph{a priori}, that gamma functions are homomorphisms of groups. However, when this holds, the situation is surprisingly simple to analyze.

\begin{proposition}[\cite{Campedel_2020}] \thlabel{thm:gamma_morphism_ker}
    Let $G$ be a finite group, $A\le G$ and $\gamma\colon A\to\aut(G)$ a function such that $A$ is invariant under $\gamma(A)$. Then, any two of the following conditions imply the third one.
    \begin{enumerate}
        \item $\gamma([A,\gamma(A)])=\{\id\}$.
        \item $\gamma\colon A\to\aut(G)$ is a group homomorphism.
        \item $\gamma$ satisfies the gamma functional equation on $A$.
    \end{enumerate}
\end{proposition}

We have described the conjugation of a gamma function under the action of automorphisms. As the following result describes, there is a special case of \emph{conjugation} with respect to the inversion map. Note that, in the abelian case, the inversion map is just one of the automorphisms already described. 

\begin{proposition}[\cite{Campedel_2020}] \thlabel{thm:gamma_inv}
    Let $G$ be a finite group, $\gamma\colon G\to\aut(G)$ be a gamma function on $G$ and denote by $N$ the associated regular subgroup and by $\circ$ the associated circle operation. Define 
    \[
    \begin{aligned}
        \overline\gamma\colon G&\to\aut(G) \\
        g&\mapsto\overline\gamma(g) =\gamma(g\mo)\iota(g\mo)
    \end{aligned}
    \]
    Then
    \begin{enumerate}
        \item The conjugate $N^\inv\le\hol(G)$ of $N$ under $\inv\in\sym(G)$ is a regular subgroup.
        \item $\overline\gamma$ is a gamma function and it is the one associated with the regular subgroup $N^\inv$.
        \item Denote by $\overline\circ$ is the circle operation associated with $\overline\gamma$. Then $\inv\colon(G,\circ)\to(G,\overline\circ)$ is a group isomorphism, therefore $N^\inv\cong N$.
    \end{enumerate}
\end{proposition}

We conclude the preliminaries with two strong results that we will use at the very beginning of the case of cyclic groups, in order to obtain some useful lemmas. 

\begin{theorem}[\cite{Caranti_2018}] \thlabel{thm:srg2}
    Let $(G,\cdot)$ be a finite group. The following data are equivalent.
    \begin{enumerate}
        \item A bi-skew brace $(G,\cdot,\circ)$.
        \item A regular subgroup $N\le\hol(G)$ which is normalized by $\rho(G)$.
        \item A gamma function $\gamma\colon G\to\aut(G)$ which satisfies $$\gamma(xy)=\gamma(y)\gamma(x)\quad\quad\quad\forall x,y\in G$$ that is $\gamma$ is an anti-homomorphism.
        \item A gamma function $\gamma\colon G\to\aut(G)$ which satisfies $$\gamma(x^{\gamma(y)})=\gamma(x)^{\gamma(y)}\quad\quad\quad\forall x,y\in G.$$
        \item A function $\gamma\colon G\to\aut(G)$ which satisfies
        \[
            \begin{cases}
                    \gamma(xy)=\gamma(y)\gamma(x) \\
                    \gamma(x^{\gamma(y)})=\gamma(x)^{\gamma(y)}
            \end{cases}
            \quad\quad\quad\forall x,y\in G.
        \]
        \item A gamma function $\gamma\colon G\to\aut(G)$ which satisfies $$\gamma([G,\overline\gamma(G)])=\{\id\}$$
        where $\overline\gamma\colon G\to\aut(G) $ is the gamma function $\overline\gamma(g)=\gamma(g\mo)\iota(g\mo)$.
        \item A gamma function $\gamma\colon G\to\aut(G)$ which satisfies $$\gamma(x\mo y\mo x^{\gamma(y)}y)=\{\id\}\quad\quad\quad\forall x,y\in G.$$
    \end{enumerate}
\end{theorem}

It is important to notice the role that the kernel of a gamma function plays with respect to both the natural operation on the group and the circle operation induced by the gamma function.

\begin{lemma}[\cite{Caranti_2020}] \thlabel{thm:operation_ker}
    Let $G$ be a finite group and let $\gamma\colon G\to\aut(G)$ be a gamma function on $G$. Denote by $\circ$ the circle operation on $G$ induced by $\g$. Then for every $x,y\in G$, if $y\in\ker\gamma$ we have
    $$x\circ y=x\cdot y$$
    that is, the natural and the circle operations on $G$ agree if the second factor is taken in $\ker\gamma$.
\end{lemma}

\begin{theorem}[\cite{Caranti_2018}] \thlabel{thm:srg3}
    Let $(G,\cdot)$ be a finite group. The following data are equivalent.
    \begin{enumerate}
        \item A bi-skew brace $(G,\cdot,\circ)$ such that $\aut(G,\cdot)\le\aut(G,\circ)$.
        \item A regular normal subgroup $N\norm\hol(G)$.
        \item A regular subgroup $N\le\hol(G)$ which is normalized by $\aut(G)$.
        \item A gamma function $\gamma\colon G\to\aut(G)$ which satisfies $$\gamma(x^\alpha)=\gamma(x)^\alpha\quad\quad\quad\forall x\in G,\,\forall\alpha\in\aut(G).$$
        \item A function $\gamma\colon G\to\aut(G)$ such that
        \[
            \begin{cases}
                \gamma(xy)=\gamma(y)\gamma(x) \\
                \gamma(x^\alpha)=\gamma(x)^\alpha
            \end{cases}
            \quad\quad\quad\forall x\in G,\,\forall\alpha\in\aut(G).
        \]
    \end{enumerate}
\end{theorem}

\section{Gamma function on cyclic groups} \label{sec:2}

This small section contains some easy consequences of the general theory of gamma functions, when restricted to the cyclic case. We will see that the behavior of a gamma function on a finite cyclic group only depends on the cardinality of its image, and that there are some special cases when such image has cardinality 2. Here and in the rest of the paper, for an additive group $G$ and an integer $k\in\Z$, we denote by $\sigma_k$ the $k$-th multiple map on $G$, that is $\sigma_k\colon G\to G, x\mapsto kx$. Moreover, we will interpret a cyclic group $\C_m$ as $\Zmod{m}$.

\begin{definition}
    Let $G=\C_m$ be a cyclic group of order $m$, let $q\in\N$ be such that $q\mid m$ and let $\gamma\colon G\to\aut(G)$ be a function on $G$. Identifying any integer with its residue class modulo $q$, we say that $\gamma$ is \textbf{defined modulo $q$} if for every $x,y\in G$ we have $\gamma(x)=\gamma(y)$ if and only if $x\equiv y\pmod{q}$.
\end{definition}

\begin{proposition} \thlabel{thm:gamma_cyclic}
    Let $G=\C_m$ be a cyclic group of order $m$ and let $\gamma\colon G\to\aut(G)$ be a gamma function on $G$. Then $\gamma$ is defined modulo $|\gamma(G)|$.
\end{proposition}

\begin{proof}
    Denote by $\circ$ the circle operation on $G$ induced by $\g$. Then, there is an isomorphism $\dfrac{(G,\circ)}{\ker\g}\cong\g(G)$, where 
    $\ker\gamma=\left\{t\cdot|\gamma(G)|\, :\, 0\le t<\dfrac{m}{|\gamma(G)|}=|\!\ker\gamma|\right\}$ is the unique cyclic subgroup of $G$ of order $\dfrac{m}{|\g(G)|}$. Then for every $x,y\in G$
    \begin{align*}
        x\equiv y\pmod{|\gamma(G)|}\implies |\gamma(G)|\mid x-y\implies x-y\in\ker\gamma
    \end{align*}
    hence $x=y+k$ for some $k\in\ker\gamma$. Therefore, because of \thref{thm:operation_ker},
    $$\gamma(x)=\gamma(y+k)=\gamma(y\circ k)=\gamma(y)\gamma(k)=\gamma(y).$$   
    Conversely, $$\gamma(x)=\gamma(y)\implies x\circ y\omo\in\ker\gamma\implies x\in(\ker\gamma)\circ y = y\circ (\ker\gamma)$$ that is $x=y\circ k=y+k$ for some $k\in\ker\gamma$, so $x\equiv y\pmod{|\gamma(G)|}$.
    By definition, $\gamma$ is defined modulo $|\gamma(G)|$.
\end{proof}

Now, the two following lemmas are straightforward.

\begin{lemma} \thlabel{thm:gamma_involution}
    Let $G=\C_{m}$ be a cyclic group of even order $m$ and let $\alpha\in\aut(G)$ be an involution. Let $\gamma\colon G\to\aut(G)$ be a function defined by $\gamma(x)=\alpha^x$. Then, $\gamma$ is a gamma function, it is defined modulo 2 and it is both a homomorphism and an anti-homomorphism of groups.
\end{lemma}

\begin{proof} 
    It is known that an automorphism $\alpha$ of a cyclic group (in additive notation) is of the form $\sigma_h$ for some $h$ coprime with $m$. In particular, such $h$ must be odd, since $m$ is even by assumption. Then, we can show that $\g$ is a gamma function by case distinction. For every $x,y\in G$
    \[
    \begin{aligned}
        x\equiv 0,\, y\equiv 0 \pmod2 &\implies \g(x)=\id, \g(y)=\id \implies \g(x^{\g(y)}+y)=\g(x+y)=\id=\id^2=\g(x)\g(y)\\
        x\equiv 0,\, y\equiv 1 \pmod2 &\implies \g(x)=\id, \g(y)=\alpha \implies \g(x^{\g(y)}+y)=\g(x^\alpha+y)=\g(hx+y)=\alpha=\id\alpha=\g(x)\g(y)\\
        x\equiv 1,\, y\equiv 0 \pmod2 &\implies \g(x)=\alpha, \g(y)=\id \implies \g(x^{\g(y)}+y)=\g(x+y)=\alpha=\alpha\id=\g(x)\g(y)\\
        x\equiv 1,\, y\equiv 1 \pmod2 &\implies \g(x)=\alpha, \g(y)=\alpha \implies \g(x^{\g(y)}+y)=\g(x^\alpha+y)=\g(hx+y)=\id=\alpha^2=\g(x)\g(y)\\
    \end{aligned}
    \]
    Since $\alpha^2=\id$, then by \thref{thm:gamma_cyclic} $\gamma$ is defined modulo 2. Moreover, again by case distinction, we prove that it is a group homomorphism. For every $x,y\in G$
    \[
    \begin{aligned}
        x\equiv 0,\, y\equiv 0 &\implies x+y\equiv 0\pmod2 \implies \gamma(x+y)=\id=\id^2=\gamma(x)\gamma(y)\\
        x\equiv 0,\, y\equiv 1 &\implies x+y\equiv 1\pmod2 \implies \gamma(x+y)=\alpha=\id\alpha=\gamma(x)\gamma(y)\\
        x\equiv 1,\, y\equiv 0 &\implies x+y\equiv 1\pmod2 \implies \gamma(x+y)=\alpha=\alpha\id=\gamma(x)\gamma(y)\\
        x\equiv 1,\, y\equiv 1 &\implies x+y\equiv 0\pmod2 \implies \gamma(x+y)=\id=\alpha^2=\gamma(x)\gamma(y)\\
    \end{aligned}
    \]
    It is also an anti-homomorphism because $G$ is abelian.
\end{proof}

\begin{lemma} \thlabel{thm:gamma_normal}
    Let $G=\C_{2^n}$ be a cyclic group of order $2^n$ and let $\gamma\colon G\to\aut(G)$ be a gamma function on $G$ that is defined modulo 2. Then the regular subgroup $N\le\hol(G)$ associated with $\gamma$ is normal in $\hol(G)$.
\end{lemma}

\begin{proof}
    We already know that every automorphism of $\C_{2^n}$ is of the form $\sigma_{2k+1}$ for some $k=0,\dots,2^{n-1}-1$. Then, since $\gamma$ is defined modulo 2 and $\aut(\C_{2^n})$ is abelian, we have
    $$\gamma(x^{\sigma_{2k+1}})=\gamma((2k+1)x)=\gamma(2kx+x)=\gamma(x)=\gamma(x)^{\sigma_{2k+1}}$$
    that is, from \thref{thm:srg3}, the regular subgroup $N\le\hol(G)$ associated with $\gamma$ is normal in $\hol(G)$.
\end{proof}

\section{The classification in the case \texorpdfstring{$p=2$}{p=2}} \label{sec:p2}

This section is the core of the work. We start with some notation and the formal definition of local normalizing graph. After that, we deduce such graph in some small cases, and finally we complete the classification in the case where $G$ is a cyclic $2-$group.

\begin{definition} We denote as follows some relevant $2$-groups
    \begin{center}
        \begin{tabular}{ lc }
        
            $\C_{2^n}=\left\langle x\, :\, x^{2^n}=1\right\rangle$ & \textbf{cyclic group} \\[0.5em]
        
            $\C_2\times\C_{2^{n-1}}=\left\langle x,y\, :\, x^{2^{n-1}}=y^2=1,\, y\mo xy=x\right\rangle$ & \textbf{direct product of cyclic groups} \\[0.5em]
        
            $\Q_{2^n}=\left\langle x,y\, :\, x^{2^{n-1}}=1,\, x^{2^{n-2}}=y^2,\, y\mo xy=x\mo\right\rangle$ & \textbf{quaternion group} \\[0.5em]
        
            $\D_{2^n}=\left\langle x,y\, :\, x^{2^{n-1}}=y^2=1,\, y\mo xy=x\mo\right\rangle$ & \textbf{dihedral group} \\[0.5em]
        
            $\SD_{2^n}=\left\langle x,y\, :\, x^{2^{n-1}}=y^2=1,\, y\mo xy=x^{2^{n-2}-1}\right\rangle$ & \textbf{semi-dihedral group} \\[0.5em]
        
            $\M_{2^n}=\left\langle x,y\, :\, x^{2^{n-1}}=y^2=1,\, y\mo xy=x^{2^{n-2}+1}\right\rangle$ & \textbf{modular group} \\[0.5em]
        
        \end{tabular}
    \end{center}
\end{definition}

\begin{definition}
    Let $G$ be a finite group. The \textbf{local normalizing graph} of $G$ is the (undirected) graph whose vertices are the regular subgroups of $\hol(G)$, and two vertices $N$ and $M$ are connected by an edge $\{N,M\}$ if and only if $N$ and $M$ mutually normalize each other in $\sym(G)$.
\end{definition}

\subsection{Small cases}

Tacitly using the important characterization in \thref{thm:srg1}, here and in the rest of the paper, we begin the case studies with some trivial cases, that is, cyclic groups with order a small power of 2. 

\begin{proposition} \thlabel{thm:c2g1}
    The map $\gamma_1\colon\C_2 \to\aut(\C_2)$ defined by
    \[
    \gamma_1(x)=\sigma_1\qquad\forall x\in\C_2
    \]    
    is a gamma function on $\C_2$, and the associated regular subgroup $N_1\le\hol(\C_2)$ is isomorphic to $\C_2$. Moreover, $N_1$ is normal in $\hol(\C_2)$ and $\gamma_1$ is the unique gamma function on $\C_2$.
\end{proposition}

\begin{proof}
    The trivial map is a gamma function because the gamma functional equation is trivially satisfied, and that the associated regular subgroup of $\hol(\C_2)$ is $\rho(\C_2)\cong\C_2$, which is also normal in $\hol(\C_2)$. The uniqueness follows from the fact that $\aut(\C_2)=\{\sigma_1\}$ is a singleton, hence there is only a unique way to construct a function whose codomain is $\aut(\C_2)$. 
\end{proof}

\begin{theorem}[Local normalizing graph of $\C_2$] \thlabel{thm:c2graph}
    The local normalizing graph of $\C_2$ is
    
    
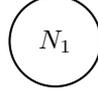
\begin{center}
        \begin{tikzpicture}[
            minimum size=\nodeSize,
            node distance=\nodeDist,
            thick, 
            main/.style = {draw=black, fill=white,  circle},
            ] 
            \node[main] (1) {$N_1$};  
        \end{tikzpicture} 
    \end{center}
\end{theorem}

\begin{proposition} \thlabel{thm:c4g12}
    The maps $\gamma_1,\gamma_2\colon\C_4 \to\aut(\C_4)$ defined by
    \[
    \gamma_1(x)=\sigma_1,\qquad\gamma_2(x)=\sigma_3^x\qquad\forall x\in\C_4
    \]
    are gamma functions on $\C_4$, and the associated regular subgroups $N_1,N_2\le\hol(\C_4)$ are isomorphic respectively to $\C_4$ and $\V$, where $\V$ denotes the Klein group. Moreover, $N_1$ and $N_2$ are normal in $\hol(\C_4)$,  $\{\gamma_1,\gamma_2\}$ are the unique gamma functions on $\C_4$, and $\{N_1,N_2\}$ are mutually normalizing regular subgroups of $\hol(\C_4)$.
\end{proposition}

\begin{proof}
    It is a direct consequence of \thref{thm:srg1} and \thref{thm:srg3}.
\end{proof}

\begin{theorem}[Local normalizing graph of $\C_4$] \thlabel{thm:c4graph}
    The local normalizing graph for $\C_4$ is
    
    
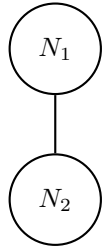
\begin{center}
        \begin{tikzpicture}[
            minimum size=\nodeSize,
            node distance=\nodeDist,
            thick, 
            main/.style = {draw=black, fill=white,  circle},
            ] 
            \node[main] (1) {$N_1$}; 
            \node[main] (2) [below of=1] {$N_2$}; 

            \draw (2) -- (1);
        \end{tikzpicture} 
    \end{center}
\end{theorem}

\subsection{General case}

Once excluded some trivial and degenerate cases, the behavior of the local normalizing graphs turns out to be surprisingly easy to describe, using the language of gamma functions.

\subsubsection{Existence problem}

\begin{proposition} \thlabel{thm:c2ng1234}
    Let $G=\C_{2^n}$ be a cyclic group of order $2^n$ with $n\ge 3$. Then the following maps $G \to\aut(G)$ are gamma functions on $G$.    
    
    \begin{table}[H]
    \centering
    {\tabulinesep=1.6mm
    \begin{tabu}{lccc}
    \hline
    \shortstack[l]{Gamma\\function} & \shortstack{Defined\\modulo} & \shortstack{Isomorphism\\class} & \shortstack{Normal\\subgroup} \\
    \hline
    $\g_1(x)=\sigma_1$ & $1$ & $\C_{2^n}$ & yes  \\
    $\g_2(x)=\sigma_{2^{n-1}+1}^x$ & $2$ & $\C_{2^n}$ & yes  \\
    $\g_3(x)=\sigma_{2^{n-1}-1}^x$ & $2$ & $\Q_{2^n}$ & yes  \\
    $\g_4(x)=\sigma_{2^{n}-1}^x$ & $2$ & $\D_{2^n}$ & yes \\[-15pt] 
    $\gamma_5(x)=\left\langle
    \begin{array}{lr}
        \sigma_1 & x\equiv 0 \pmod4\\
        \sigma_{2^{n-1}-1} & x\equiv 1 \pmod4\\
        \sigma_{2^{n-1}+1} & x\equiv 2 \pmod4\\
        \sigma_{2^{n}-1} & x\equiv 3 \pmod4\\
    \end{array}
    \right.$     
    & $4$ & $\SD_{2^n}$ & no \\[-25pt]
    $\gamma_6(x)=\left\langle
    \begin{array}{lr}
        \sigma_1 & x\equiv 0 \pmod4\\
        \sigma_{2^{n}-1} & x\equiv 1 \pmod4\\
        \sigma_{2^{n-1}+1} & x\equiv 2 \pmod4\\
        \sigma_{2^{n-1}-1} & x\equiv 3 \pmod4\\
    \end{array}
    \right.$
    & $4$ & $\SD_{2^n}$ & no \\[-15pt]
    \hline
\end{tabu}}
\end{table}
\end{proposition}

\begin{proof}
    For $\g_1$ the proof is trivial, and similar to that of \thref{thm:c2g1}. To show that $\g_2$ is a gamma function, we first use \thref{thm:gamma_involution}, and then the equivalent conditions of \thref{thm:gamma_morphism_ker}. Denote by $\circ$ and by $N_2$ the circle operation and the regular subgroup of $\hol(G)$ associated with $\g_2$. Observe that the latter is normal in $\hol(G)$ because of \thref{thm:gamma_normal}. To determine its isomorphism class, we find directly the generators of a presentation of $(G,\circ)$. Since the powers of $1$ with respect to the circle operation $\circ$ are
     \begin{equation}
        1^{\circ k}=\left\langle
            \begin{array}{lr}
                k & \quad\mbox{if}\quad k\equiv 0,1 \pmod4\\
                2^{n-1}+k & \quad\mbox{if}\quad k\equiv 2,3 \pmod4
            \end{array}
        \right.\quad\forall k\in\N
    \end{equation}
    it is easy to conclude that $(G,\c)=\left\langle r\, :\,r^{\c 2^n}=0\right\rangle$ is a cyclic group of order $2^n$, where $r=1$. 
    
    In a similar way we prove the claim for all the other functions in the table. To show that they are gamma functions, we use both \thref{thm:gamma_involution} and \thref{thm:gamma_morphism_ker} for $\g_3$ and $\g_4$, and we proceed by a direct verification for $\g_5$ and $\g_6$. To show that $\g_5$ and $\g_6$ are associated with a regular subgroup which is not normal we observe that they are not anti-homomorphisms of groups and apply \thref{thm:srg3}. In particular, we report in the following table some information about the generators of all the associated skew braces.
    
    \def\arraystretch{1.5}%
    \begin{table}[h]
        \centering
        \begin{tabular}{|c|c|c|c|}
            \hline
             Regular subgroup & Isomorphism class & Generators of $(G,\circ)$ & Relations of $(G,\circ)$ \\
            \hline
            
            \multirow{ 3}{*}{$N_3$} & \multirow{ 3}{*}{$\Q_{2^n}$} & \multirow{ 3}{*}{$1,2$} & $2^{\circ k}=2k$ \\
            & & & $1\c 1=2^{n-1}$ \\
            & & & $1\omo\c 2\c 1 =2\omo$ \\
            
            \hline
            
            \multirow{ 3}{*}{$N_4$} & \multirow{ 3}{*}{$\D_{2^n}$} & \multirow{ 3}{*}{$1,2$} & $2^{\circ k}=2k$ \\
            & & & $1\c 1=0$ \\
            & & & $1\omo\c 2\c 1 =2\omo$ \\
            
            \hline
            
            \multirow{ 3}{*}{$N_5$} & \multirow{ 3}{*}{$\SD_{2^n}$} & \multirow{ 3}{*}{$2,3$} & $2^{\circ k}=2k$ \\
            & & & $3\c 3=0$ \\
            & & & $3\omo\c 2\c 3 =2^{\c (2^{n-2}-1)}$ \\
            
            \hline
        \end{tabular}
    \end{table}
    
    Once known the isomorphism class of $\g_5$, there is a shortcut for $\g_6$: we conjugate the gamma function $\g_5$ under the automorphism $\inv\in\aut(G)$ as in \thref{thm:gamma_alpha_lemma}, from which we also know that the isomorphism class of the associated regular subgroup does not change. \qedhere
\end{proof}

To show the existence of the remaining gamma functions, we proceed in a similar manner. The following standard arithmetical facts will be useful (the proofs are omitted).

\begin{lemma} \thlabel{thm:five_lemma_mod_1}
    For every $n\in\N,n\ge 4$ and every $k\in\N, 1\le k \le 2^{n-1}$ we have $$\frac{5^k-1}{2}\equiv 0\pmod{2^n}\iff k=2^{n-1}.$$
\end{lemma}

\begin{lemma} \thlabel{thm:five_lemma_mod_2}
    For every $n\in\N,n\ge 4$, we have $$\frac{5^{2^{n-2}+1}-1}{2}\equiv 2^{n-1}+2\pmod{2^n}.$$
\end{lemma}

\begin{lemma} \thlabel{thm:five_lemma_mod_3}
    For every $n,k,u\in\N$ such that $n\ge 4,\ 2\le u < n$ and $1\le k \le 2^n$ we have $$2^{-u}\left[(2^u+1)^k-1\right]\equiv 0\pmod{2^n}\iff k=2^n.$$
\end{lemma}

\begin{proposition} \thlabel{thm:c2ngpmc}
    Let $G=\C_{2^n}$ be a cyclic group of order $2^n$ with $n\ge 4$. Then the following maps $G \to\aut(G)$ are gamma functions on $G$.
    \begin{table}[H]
    \centering
    {\tabulinesep=1.6mm
    \begin{tabu}{lccc}
        \hline
        \shortstack[l]{Gamma\\function} & \shortstack{Defined\\modulo} & \shortstack{Isomorphism\\class} & \shortstack{Normal\\subgroup} \\
        \hline
        $\g_{\mathsf{p}}(x)=\sigma_{2x+1}$ & $2^{n-1}$ & $\C_2\times\C_{2^{n-1}}$ & no  \\
        $\gamma_\mathsf{m}(x)=\left\langle
        \begin{array}{lr}
            \sigma_{2x+1} & x\equiv 0 \pmod2\\
            \sigma_{2x+2^{n-2}+1} & x\equiv 1 \pmod2
        \end{array}
        \right.$  
        & $2^{n-1}$ & $\M_{2^{n}}$ & no \\ 
        $\g_{\mathsf{c},u}(x)=\sigma_{2^ux+1},\qquad 2\le u\le n-2$ & $2^{n-u}$ & $\C_{2^{n}}$ & no \\
        \hline
    \end{tabu}}
    \end{table}
\end{proposition}

\begin{proof}
    The proof is similar to that of \thref{thm:c2ng1234}. We prove that they are gamma functions by a direct verification and we exhibit the generators of the associated regular subgroup. Moreover, we use \thref{thm:srg3} again to show that such subgroups are not normal. We report as above some information about the generators, whose orders are determined by using \thref{thm:five_lemma_mod_1}, \thref{thm:five_lemma_mod_2}, and \thref{thm:five_lemma_mod_3}.

    \def\arraystretch{1.5}%
    \begin{table}[h]
        \centering
        \begin{tabular}{|c|c|c|c|}
            \hline
             \shortstack{Regular\\subgroup}& \shortstack{Isomorphism\\class} & \shortstack{Generators\\of $(G,\circ)$} & \shortstack{Relations\\of $(G,\circ)$} \\
            \hline
            
            \multirow{ 3}{*}{$N_\mathsf{p}$} & \multirow{ 3}{*}{$\C_{2}\times\C_{2^{n-1}}$} & \multirow{ 3}{*}{$2, 2^n-1$} & $2^{\c k}=\frac{5^k-1}{2}$ \\
            & & & $(2^n-1)\c (2^n-1)=0$ \\
            & & & $(2^n-1)\omo\c 2\c (2^n-1)=2$ \\
            
            \hline
            
            \multirow{ 3}{*}{$N_\mathsf{m}$} & \multirow{ 3}{*}{$\M_{2^n}$} & \multirow{ 3}{*}{$2, 2^{n-2}+2^{n-3}-1$} & $2^{\c k}=\frac{5^k-1}{2}$ \\
            & & & $(2^{n-2}+2^{n-3}-1)^{\c 2} = 0$ \\
            & & & $(2^{n-2}+2^{n-3}-1)\omo\c 2\c (2^{n-2}+2^{n-3}-1)= 2^{n-1}+2$ \\
            
            \hline
            
            $N_{\mathsf{c},u}$ & $\C_{2^n}$ & $1$ & 
            $1^{\c k}=2^{-u}\left[(2^u+1)^k-1\right]$\\            
            \hline
            
        \end{tabular}
    \end{table}    
\end{proof}

\begin{lemma} \thlabel{thm:k_lemma}
    Let $G=\C_{2^n}$ be a cyclic group of order $2^n$ with $n\ge 3$, let $N\le\hol(G)$ be a regular subgroup, and let $\gamma\colon G\to\aut(G)$ be the associated gamma function. Then the cardinality $|\nor_{\sym(G)}(N)\cap\aut(G)|=|K|$ where $$K=\left\{k\in\Z\, :\, 0\le k < 2^{n-1},\ \g(x)=\g((2k+1)x)\quad \forall x\in G\right\}.$$
\end{lemma}

\begin{proof}
    To compute the number $|\nor_{\sym(G)}(N)\cap\aut(G)|$, it is enough to determine the cardinality of the stabilizer of $\gamma$ in $\aut(G)$ under conjugation. Thus $\g^\alpha(x)=\g(x)\ \forall x\in G\iff\alpha\mo\g(x^{\alpha\mo})\alpha=\g(x)\ \forall x\in G\iff \g(x^{\alpha\mo})=\g(x)\ \forall x\in G\iff \g(x)=\g(x^{\alpha})\ \forall x\in G$.
    Since each element $\alpha\in\aut(G)$ is of the form $\alpha=\sigma_{2k+1}$ for some $0\le k <2^{n-1}$, we obtain the conclusion.
\end{proof}

We are ready to state and prove the main result concerning the existence of regular subgroups of $\hol(G)$, related to the sizes of the conjugacy classes of them.

\begin{proposition} \thlabel{thm:c2nclasses}
    Let $G=\C_{2^n}$ be a cyclic group of order $2^n$ with $n\ge 3$. Under the action by conjugation of $\aut(G)$ on the family of regular subgroups of $\hol(G)$, we have
    \begin{enumerate}
        \item Four conjugacy classes of size $1$: $\{N_1\},\{N_2\},\{N_3\},$ and $\{N_4\}$.
        \item One conjugacy class of size $2$: $\{N_5,N_6\}$.
        \item Two conjugacy classes of size $2^{n-2}$: $\{N_\mathsf{p}^\alpha\, :\,\alpha\in\aut(G)\}$, and $\{N_\mathsf{m}^\alpha\, :\,\alpha\in\aut(G)\}$ (if $n\ge 4$).
        \item Conjugacy classes of sizes $2,4,8,\dots,2^{n-3}$: $\{N_{\mathsf{c},u}^\alpha\, :\,\alpha\in\aut(G)\}$ for $u=2,\dots,n-2$, each of size $2^{n-u-1}$, for $u\le n-2$ (if $n\ge 4$).
    \end{enumerate}
\end{proposition}

\begin{proof}
    The first claim follows from the fact that $N_i\norm\hol(G)$ for $i=1,\dots 4$. The other claims are a direct consequence of \thref{thm:k_lemma} and \thref{thm:gamma_counts}, for instance, for $N_5$ and $N_6$, we have $\g(x)=\g((2k+1)x)\iff2 kx\equiv 0\pmod4$, thus $K=\left\{2t\, :\,0\le t\le 2^{n-2}-1\right\}$ has cardinality $|K|=2^{n-2}$, thus the corresponding conjugacy class of gamma functions has cardinality 2.
\end{proof}

Finally, by adding them up, we obtain the following.

\begin{corollary} \thlabel{thm:c2nclassescor}
    Let $G=\C_{2^n}$ be a cyclic group of order $2^n$. If $n=3$ then there are at least six regular subgroups in $\hol(G)$. If $n\ge 4$, then there are at least $3\cdot 2^{n-2}+4$ regular subgroups of $\hol(G)$.
\end{corollary}

\subsubsection{Uniqueness problem}

The goal of this subsection is to prove that there are no regular subgroups in $\hol(G)$, other than those already found above. Roughly speaking, we aim to obtain a result like \thref{thm:c2nclassescor} where the words \emph{“at least"} are replaced by \emph{“exactly"}. We will exploit a result of N. P. Byott found in \cite{byott1996hopfgalois}.

\begin{proposition} \thlabel{thm:c2n!1234}
    There are exactly four regular normal subgroups of $\hol(G)$, namely they are $N_1, N_2, N_3, N_4$.
\end{proposition}

\begin{proof}
    Let $N\le\hol(G)$ be a regular normal subgroup, and let $\gamma\colon G\to\aut(G)$ the associated gamma function. Because of \thref{thm:srg3} and since $\aut(G)$ is abelian, the fact that $N$ is regular normal is equivalent to
    \begin{equation} \label{eqn:gamma_normal}
    \begin{cases}
        \gamma(x+y)=\gamma(x)\gamma(y) \\
        \gamma(x^\alpha)=\gamma(x)
    \end{cases}\quad\forall x,y\in G,\forall\alpha\in\aut(G)
    \end{equation}
    In other words, $\gamma\colon G\to\aut(G)$ is a group homomorphism (with respect to the natural additive operation on $G$) and the value of $\gamma(x)$ does not change if we apply any automorphism on $x$. From this data, it is easy to see that $\g(2x)=\id$ for every $x\in G$, in particular, $\g(1)^2=\id$, that is $\g(1)$ is either the identity map or an involution of $\aut(G)$.
\end{proof}

\begin{corollary} \thlabel{thm:gamma_image_c1c2}
    Let $\gamma\colon G\to\aut(G)$ be a gamma function on $G$ such that $\gamma(G)$ is isomorphic to either $\C_1$ or $\C_2$. Then $\gamma$ is associated with a regular normal subgroup $N\norm\hol(G)$ such that $N\cong\C_{2^n},\Q_{2^n},\D_{2^n}$.
\end{corollary}

\begin{proof}
    If $\gamma\colon G\to\aut(G)$ is a gamma function such that $|\gamma(G)|=1$, then it is the trivial map $\g_1$ associated with a cyclic regular subgroup. Otherwise, if $|\gamma(G)|=2$, thanks to \thref{thm:gamma_cyclic}, $\gamma$ is defined modulo 2, and it is associated with a regular normal subgroup of $\hol(G)$, because of \thref{thm:gamma_normal}. Now, from \thref{thm:c2n!1234} we conclude that $\g$ must be associated with either $\g_2$, $\g_3$ or $\g_4$.
\end{proof}

Let us recall Theorem 4.4 of \cite{gorenstein2007finite}.

\begin{theorem}[Zassenhaus] \thlabel{thm:hall_zassenhaus}
    Let $n\ge 1$ and let $p$ be a prime. Let $G$ be a group of order $p^n$ with a cyclic maximal subgroup $M\cong\C_{p^{n-1}}$. 
    
    \noindent If $p=2$, then $G$ belongs to exactly one of the following six isomorphism classes.
    \begin{enumerate}
        \item $G\cong\C_{2^n}$ where the isomorphism sends $M$ to the subgroup of multiples of $2$.
        \item $G\cong\C_2\times\C_{2^{n-1}}$ with $M$ being the second direct factor (for $n\ge 2$).
        \item $G\cong\Q_{2^n}$ (for $n\ge 3$).
        \item $G\cong\D_{2^n}$ that is the semidirect product $\C_2\ltimes M$ where $\C_2$ acts on $M$ via multiplication by $-1$ (for $n\ge 3$).
        \item $G\cong\SD_{2^n}$ that is the semidirect product $\C_2\ltimes M$ where $\C_2$ acts on $M$ via multiplication by $2^{n-2}-1$ (for $n\ge 4$).
        \item $G\cong\M_{2^n}$ that is the semidirect product of $\C_2\ltimes M$ where $\C_2$ acts on $M$ via multiplication by $2^{n-2}+1$ (for $n\ge 4$).
    \end{enumerate}
    If $p$ is an odd prime, then $G$ belongs to exactly one of the following three isomorphism classes.
    \begin{enumerate}
        \item $G\cong\C_{p^n}$ where the isomorphism sends $M$ to the subgroup of multiples of $p$.
        \item $G\cong\C_p\times\C_{p^{n-1}}$ with $M$ being the second direct factor (for $n\ge 2$).
        \item $G\cong\C_p\ltimes M$ where $\C_p$ acts on $M$ via multiplication by $p^{n-2}+1$ (for $n\ge 3$).
    \end{enumerate}
\end{theorem}

\begin{proposition} \thlabel{thm:zass_lemma}
    Every regular subgroup of $\hol(G)$ has a cyclic maximal subgroup.
\end{proposition}

\begin{proof}
    Let $N\le\hol(G)$ be a regular subgroup. Denote by $\gamma\colon G\to\aut(G)$ its associated gamma function and by $\circ$ the associated circle operation on $G$ such that $(G,\circ)\cong N$. Since $|G|=2^n$ and $|\!\aut(G)|=2^{n-1}$, $\g$ cannot be injective, thus there is a non-trivial element $k\in\ker\gamma$. It is not restrictive to assume that $k$ is even. Define
    $$\overline{M}=\{x\circ k\, :\,x\in G,\, x\mbox{ even }\}=\{x+k\, :\,x\in G,\, x\mbox{ even }\}$$ and thanks to the isomorphism $\nu\colon(G,\circ)\to N$ it is easy to see that $M=\nu(\overline{M})\le N$ is a cyclic maximal subgroup of $N$.
\end{proof}

\begin{corollary} \thlabel{thm:classification_regular_2}
    Every regular subgroup of $\hol(G)$ belongs to exactly one of the following six isomorphism classes: $\C_{2^n}$, $\Q_{2^n}$, $\D_{2^n}$, $\SD_{2^n}$, $\M_{2^n}$, $\C_2\times\C_{2^{n-1}}$. 
\end{corollary}

\begin{proof}
    It follows from \thref{thm:hall_zassenhaus} and \thref{thm:zass_lemma}.
\end{proof}

\thref{thm:classification_regular_2} restricts the eligible isomorphism types of regular subgroups of $\hol(G)$. We continue our approach to the \emph{uniqueness problem} in two different ways. For cyclic groups, we use a powerful result of N. P. Byott (see \cite{byott1996hopfgalois}), which allows us to determine without counting how many subgroups of that kind there are. Subsequently, we proceed with a direct proof in the language of skew braces for the remaining cases. For any $2-$group $G$ we denote by $\Omega_1(G)$ the subgroup of $G$ generated by the involutions. It is well known that $\Omega_1(\aut(\C_{2^n}))=\{\sigma_1,\sigma_{2^{n-1}-1},\sigma_{2^{n-1}+1},\sigma_{2^{n}-1}\}\cong\V$, which will be exploited later.

\begin{proposition}[Byott, \cite{byott2006cyclic}]
    There are exactly $2^{n-2}$ regular subgroups of $\hol(G)$ isomorphic to $\C_{2^n}$.
\end{proposition}

\begin{proposition} \thlabel{thm:c2n!56}
    There are exactly $2$ regular subgroups of $\hol(G)$ isomorphic to $\SD_{2^n}$.
\end{proposition}

\begin{proof}
    Let $N\le\hol(G)$ be a regular subgroup isomorphic to $\SD_{2^n}$. Let $\gamma\colon G\to\aut(G)$ be the gamma function associated with $N$ and denote by $\circ$ the induced circle operation on $G$ such that $$(G,\c)=\left\langle r,s\, :\,r^{\c 2^{n-1}}=s^{\c 2}=0,\ s\omo\c r\c s=r^{\c (2^{n-2}-1)}\right\rangle\cong\SD_{2^n}.$$
    It is easy to see that $\g(x)^2=\id$ for every $x\in G$. Notice that we need to have $\g(r)\neq\id,\g(s)\neq\id,\g(r)\neq\g(s)$,
    otherwise we would have $\gamma(G)\cong\C_1,\C_2$, which is not possible because of \thref{thm:gamma_image_c1c2}.    
    Moreover, since $\gamma(G)\le\Omega_1(\aut(G))$, and it contains at least three different elements, we conclude that $\gamma(G)=\Omega_1(\aut(G))$. We have, a priori, six possibilities for the gamma function $\gamma$, but we are going to conclude that there are, in fact, only two. 
    Since $r$ is a generator of the maximal cyclic subgroup of index 2 of $(G,\c)$, because of \thref{thm:zass_lemma}, it is not restrictive to assume that it is even. Moreover, 
    \begin{equation} \label{eqn:c2ng56_1}
        r^{\c 2} = r^{\gamma(r)}+r\not\equiv 0 \pmod{2^n}
    \end{equation}
    because its order is $2^{n-1}\neq 2$, for $n\ge 3$. Now observe that, because of \eqref{eqn:c2ng56_1} we have only one possibility for $\g(r)$.  Indeed, if $\g(r)=\s_{2^n-1}$, then
    \[
    \begin{aligned}
        r^{\c 2} = r^{\gamma(r)}+r = (2^n-1)r+r= 2^nr\equiv 0 \pmod{2^n}
    \end{aligned}
    \]
    in contradiction with \eqref{eqn:c2ng56_1}. In the same way, if $\g(r)=\s_{2^{n-1}-1}$, then
    \[
    \begin{aligned}
        r^{\c 2} = r^{\gamma(r)}+r =(2^{n-1}-1)r+r= 2^{n-1}r\equiv 0 \pmod{2^n}
    \end{aligned}
    \]
    because $r$ is even, again in contradiction with \eqref{eqn:c2ng56_1}. Therefore, we are forced to set $\g(r)=\s_{2^{n-1}+1}$ and by defining either $\gamma(s)=\s_{2^{n-1}-1}$ or $\gamma(s)=\s_{2^n-1}$, we obtain the conclusion.
\end{proof}

\begin{corollary} \thlabel{thm:gamma_image_v}
    Let $\gamma\colon G\to\aut(G)$ be a gamma function on $G$ such that $\gamma(G)\cong\V$. Then $\gamma$ is associated with a regular subgroup $N\le\hol(G)$ such that $N\cong\SD_{2^n}$.
\end{corollary}

\begin{proof}
    Let $N\le\hol(G)$ be the regular subgroup associated with $\gamma$ and let $\circ$ the induced circle operation on $G$ such that $(G,\circ)\cong N$. By the hypothesis, we know that $\g(x)^2=\id$ for every $x\in G$. From \thref{thm:zass_lemma}, let $r$ be a generator of the cyclic maximal subgroup $M\le N$ of index 2. Since $M\neq N$, let $s\in N\setminus M$ be such that $\g(s)\neq\id$ and $\g(s)\neq\g(r)$. Such an element $s$ must exist, otherwise we would have $\gamma(G)\cong\C_1,\C_2$, which is impossible. Since $$r^{\c 2} = r^{\gamma(r)}+r\not\equiv 0 \pmod{2^n}$$ because the order of $r$ is $2^{n-1}\neq 2$, arguing as in the proof of \thref{thm:c2n!56}, we obtain that $\g(r)=\s_{2^{n-1}+1}$ and either $\gamma(s)=\s_{2^{n-1}-1}$ or $\gamma(s)=\s_{2^n-1}$, that is the conclusion.
\end{proof}

\begin{proposition}
    There is exactly $1$ regular subgroup of $\hol(G)$ isomorphic to $\Q_{2^n}$.
\end{proposition}

\begin{proof}
    Let $N\le\hol(G)$ be a regular subgroup isomorphic to $\Q_{2^n}$. Let $\gamma\colon G\to\aut(G)$ be the gamma function associated with $N$ and denote by $\circ$ the induced circle operation on $G$ such that $$(G,\c)=\left\langle r,s\, :\,r^{\c 2^{n-2}}=s^{\c 2},\ r^{\c 2^{n-1}}=0,\ s\omo\c r\c s=r\omo\right\rangle\cong\Q_{2^n}.$$
    First of all, we want to prove that $\g(x)^2=\id$ for every $x\in G$. Since $\aut(G)$ is abelian and because of the previous presentation, it is enough to prove that $\g(r)^2=\g(s)^2=\id$. Since $\gamma\colon (G,\circ)\to\aut(G)$ is a group homomorphism, and again because $\aut(G)$ is abelian, we have 
    $$\g(s\omo\c r\c s)=\g(r\omo)\implies\g(s)\mo\g(r)\g(s)=\g(r)\mo\implies\g(r)^2=\id$$
    $$\g(s)^2=\g(s^{\c 2})=\g(r^{\c 2^{n-2}})=\g(r)^{2^{n-2}}=\id$$
    thus $\g(x)^2=\id$ for every $x\in G$, that is $\gamma(G)\le\Omega_1(\aut(G))\cong\V$. Observe that $|\g(G)|\neq 1$, otherwise we would have $\g=\g_1$, and we already know that $\g_1$ is associated with a cyclic regular subgroup of $\hol(G)$. In the same way, $|\g(G)|\neq 4$, otherwise, from \thref{thm:gamma_image_v}, we would have $\gamma$ associated with a semidihedral regular subgroup of $\hol(G)$. Thus we need to have $|\g(G)|=2$. Now, because of \thref{thm:gamma_cyclic}, $\g$ is defined modulo 2, and then, from \thref{thm:gamma_normal}, $N$ is associated with a regular normal subgroup of $\hol(G)$, but, from \thref{thm:c2n!1234}, there is a unique regular normal subgroup $N$ isomorphic to $\Q_{2^n}$, which is $N_3$. 
\end{proof}
    
\begin{proposition}
    There is exactly $1$ regular subgroup of $\hol(G)$ isomorphic to $\D_{2^n}$.
\end{proposition}    
    
\begin{proof}
    Let $N\le\hol(G)$ be a regular subgroup isomorphic to $\D_{2^n}$. Let $\gamma\colon G\to\aut(G)$ be the gamma function associated with $N$ and denote by $\circ$ the induced circle operation on $G$ such that $$(G,\c)=\left\langle r,s\, :\,r^{\c 2^{n-1}}=s^{\c 2}=0,\ s\omo\c r\c s=r\omo\right\rangle\cong\D_{2^n}.$$
    First of all, we want to prove that $\g(x)^2=\id$ for every $x\in G$. Since $\aut(G)$ is abelian and because of the previous presentation, it is enough to prove that $\g(r)^2=\g(s)^2=\id$. Since $\gamma\colon (G,\circ)\to\aut(G)$ is a group homomorphism, and again because $\aut(G)$ is abelian, we have 
    $$\g(s\omo\c r\c s)=\g(r\omo)\implies\g(s)\mo\g(r)\g(s)=\g(r)\mo\implies\g(r)^2=\id$$
    $$\g(s)^2=\g(0)=\id$$
    thus $\g(x)^2=\id$ for every $x\in G$, that is $\gamma(G)\le\Omega_1(\aut(G))\cong\V$. Observe that $|\g(G)|\neq 1$, otherwise we would have $\g=\g_1$, and we already know that $\g_1$ is associated with a cyclic regular subgroup of $\hol(G)$. In the same way, $|\g(G)|\neq 4$, otherwise, from \thref{thm:gamma_image_v}, we would have $\gamma$ associated with a semidihedral regular subgroup of $\hol(G)$. Thus we need to have $|\g(G)|=2$. Now, because of \thref{thm:gamma_cyclic}, $\g$ is defined modulo 2, and then, from \thref{thm:gamma_normal}, $N$ is associated with a regular normal subgroup of $\hol(G)$, but, from \thref{thm:c2n!1234}, there is a unique regular normal subgroup $N$ isomorphic to $\D_{2^n}$, which is $N_4$.

\end{proof}

\begin{lemma} \thlabel{thm:surjective_gamma_lemma}
    Let $\gamma\colon G\to\aut(G)$ be a gamma function associated with a regular subgroup $N$ isomorphic to either $\C_2\times\C_{2^{n-1}}$ or $\M_{2^n}$. Then $\gamma$ is surjective.
\end{lemma}

\begin{proof}
    Because of the complete description of isomorphism classes among regular subgroups of $\hol(G)$ of \thref{thm:classification_regular_2}, we show that every non-surjective gamma function $\gamma\colon G\to\aut(G)$ on $G$ is associated with a regular subgroup isomorphic to $\C_{2^n},\Q_{2^n},\D_{2^n},\SD_{2^n}$.  
    Let $\gamma\colon G\to\aut(G)$ be a non-surjective gamma function, let $N\le\hol(G)$ be its associated regular subgroup, and denote by $\circ$ the induced circle operation of $G$. Because of \thref{thm:gamma_image_c1c2}, we know that if $\gamma(G)\cong\C_1,\C_2$, then $N$ is isomorphic to either $\Q_{2^n},\D_{2^n}$ or one of the regular normal $\C_{2^n}$. In the same way, from \thref{thm:gamma_image_v}, if $\gamma(G)\cong\V$, we already know that $N\cong\SD_{2^n}$. Therefore, since $\gamma(G)\le\aut(G)$ and $|\!\aut(G)|=2^{n-1}$ we may assume that 
    \[
        \begin{cases}
            \gamma(G)\not\cong\V \\
            4\le|\gamma(G)|\le 2^{n-2}
        \end{cases}
    \]
    To conclude the proof it is enough to show that $N\cong\C_{2^n}$. Denote by $|\gamma(G)|=2^{n-u}$ for some $u\in\{2,\dots,n-2\}$ and note that, because of \thref{thm:gamma_cyclic}, $\gamma$ is defined modulo $2^{n-u}$. Consider $\gamma(1)\in\aut(G)$, and let $k\in\{0,\dots,2^{n-1}-1\}$ be such that $\gamma(1)=\sigma_{2k+1}$. We need to have $\gamma(1)^{2^{n-u}}=\id$, that is $\sigma_{2k+1}^{2^{n-u}}=\sigma_1$. By induction, we prove that \begin{equation} \label{eqn:surjective_gamma_lemma_3}
        1^{\c t}=\sum_{i=0}^{t-1} 1^{\gamma(1)^i}=\frac{(2k+1)^t-1}{2k}\quad\forall t\in\N.
    \end{equation}
    This information, together with \thref{thm:five_lemma_mod_3}, is enough to conclude that
    \[
    \begin{aligned}
        1^{\c t}=\frac{(2k+1)^i-1}{2k}\equiv 0\pmod{2^n} \iff t=2^n
    \end{aligned}
    \]
    so the element $1\in G$ has order $2^n$ with respect to $\circ$, therefore $(G,\circ)\cong N\cong\C_{2^n}$.
\end{proof}

\begin{proposition} \thlabel{thm:c2n!pm}
    There are exactly $2^{n-2}$ regular subgroups of $\hol(G)$ isomorphic to $\C_2\times\C_{2^{n-1}}$, and exactly $2^{n-2}$ regular subgroups of $\hol(G)$ isomorphic to $\M_{2^n}$.
\end{proposition}

\begin{proof}
    We deal with the cases $\C_2\times\C_{2^{n-1}}$ and $\M_{2^n}$ at the same time, in particular, we prove that the total number of regular subgroups of $\hol(G)$ isomorphic to either $\C_2\times\C_{2^{n-1}}$ or $\M_{2^n}$ is $2^{n-1}$. Let $N\le\hol(G)$ be a regular subgroup isomorphic to either $\C_2\times\C_{2^{n-1}}$ or $\M_{2^n}$. Let $\gamma\colon G\to\aut(G)$ be the gamma function associated with $N$ and denote by $\circ$ the induced circle operation on $G$ such that $(G,\circ)\cong N$. Because of \thref{thm:surjective_gamma_lemma}, $\gamma\colon(G,\circ)\to\aut(G)$ is a surjective group homomorphism, thus there is an isomorphism 
    \begin{equation} \label{eqn:c2n!pm_1}
        \begin{aligned}
            \psi\colon\frac{(G,\circ)}{\ker\gamma}\to\aut(G) \\
        \end{aligned}
    \end{equation}
    defined by $\psi((\ker\gamma)\circ x)=\gamma(x)$ for every $x\in G$. From \thref{thm:gamma_cyclic} we know that $\gamma$ is defined modulo $2^{n-1}=|\gamma(G)|$, then an isomorphism of the form \eqref{eqn:c2n!pm_1} uniquely determines the gamma function $\gamma$. Moreover, we know that
    \begin{equation} \label{eqn:c2n!pm_2}
        \C_2\times\C_{2^{n-2}}\cong\aut(G)\cong\frac{(G,\circ)}{\ker\gamma}
    \end{equation}
    which implies that, instead of counting automorphisms of the form \eqref{eqn:c2n!pm_1}, we can conclude by counting the isomorphisms of $\aut(G)$ into itself, that is the automorphisms of $\aut(G)$. Hence
    $$|\!\aut(\aut(G))|=|\!\aut(\C_2\times\C_{2^{n-2}})|=2^{n-1}=2^{n-2}+2^{n-2}.$$
    This implies that there are at most $2^{n-1}$ regular subgroups of $\hol(G)$ isomorphic to either $\C_2\times\C_{2^{n-1}}$ or $\M_{2^n}$ is $2^{n-1}$, and since we already know that there are at least $2^{n-2}$ of each kind, so the conclusion follows.
\end{proof}

\subsubsection{Mutual normalization problem}

So far we established the existence and uniqueness of the vertices of the local normalizing graph of a cyclic $2-$group. In this section, we prove the existence and uniqueness of the edges of the graph, that is, we highlight all and the only pairs of regular subgroups of $\hol(G)$ that mutually normalize each other. We start with the most important and general result of this work, and then we apply it to obtain the answer in the cyclic group case.

\begin{theorem} \thlabel{thm:neo}
    Let $(G,\cdot)$ be a group such that $\aut(G)$ is abelian, and let $N,M\le\hol(G)$ be regular subgroups. Denote by $\gamma\colon G\to\aut(G)$ and $\delta\colon G\to\aut(G)$ the gamma functions, and by $\circ$ and $\bullet$ the circle operations associated with $N$ and $M$, respectively. Then
    $N$ and $M$ mutually normalize each other if and only if
    \begin{equation} \label{eqn:neo}
        \begin{cases}
            \gamma(h)=\gamma\left(h\cdot(g\circ h)\mo \cdot (h\bullet g)\right) \\
            \delta(h)=\delta\left(h\cdot(g\bullet h)\mo \cdot (h\circ g)\right) \\
        \end{cases}\quad\forall g,h\in G.
    \end{equation}
\end{theorem}

\begin{proof}
    Denote by $\nu\colon(G,\circ)\to N$ and $\mu\colon(G,\bullet)\to M$ the isomorphisms associated with $N$ and $M$, respectively.  Then $N$ normalizes $M$ if for every $n\in N$ and every $m\in M$, we have $n\mo m n\in M$, that is, if and only if for all $g,h\in G$ there exists $u\in G$ such that $\nu(g)\mo\mu(h)\nu(g)=\mu(u)$. Hence
    \begin{align*}
        \nu(g)\mo\mu(h)\nu(g)=\mu(u) &\iff (\gamma(g)\rho(g))\mo(\delta(h)\rho(h))(\gamma(g)\rho(g))=\delta(u)\rho(u) \\
        &\iff \rho(g)\mo\delta(h)\gamma(g)\mo\rho(h)\gamma(g)\rho(g)=\delta(u)\rho(u) \\
        &\iff \delta(h)\rho(h^{\gamma(g)})\rho(g)=(\delta(u)\delta(u)\mo)\rho(g)\delta(u)\rho(u) \\
        &\iff \delta(h)\rho(h^{\gamma(g)}\cdot g)=\delta(u)\rho(g^{\delta(u)}\cdot u)
    \end{align*}    
    However, we know that the representation of an element in $\hol(G)=\aut(G)\rho(G)$ is unique, hence the last statement holds if and only if for all $g,h\in G$
    \begin{align*}
        &\begin{cases}
        \delta(h) = \delta(u) \\
        \rho(h^{\gamma(g)}\cdot g) = \rho(g^{\delta(u)}\cdot u) \\
    \end{cases}
    \iff
        \begin{cases}
            \delta(h) = \delta(u) \\
            h^{\gamma(g)}\cdot g = g^{\delta(h)}\cdot u
        \end{cases}
    \iff
        \begin{cases}
            \delta(h) = \delta(u) \\
            (g^{\delta(h)}\cdot h\cdot h\mo)\mo\cdot h^{\gamma(g)}\cdot g =  u
        \end{cases}
    \\[1em]&\iff
        \begin{cases}
            \delta(h) = \delta(u) \\
            h\cdot(g^{\delta(h)}\cdot h)\mo\cdot h^{\gamma(g)}\cdot g =  u
        \end{cases}
    \iff
        \begin{cases}
            \delta(h) = \delta(u) \\
            u = h\cdot(g\bullet h)\mo \cdot (h\circ g)
        \end{cases}
    \iff
        \delta(h)=\delta\left(h\cdot(g\bullet h)\mo \cdot (h\circ g)\right).
    \end{align*} 
    Symmetrically, the fact that $M$ normalizes $N$ is equivalent to the condition $\gamma(h)=\gamma\left(h\cdot(g\circ h)\mo \cdot (h\bullet g)\right)$, for every $g,h\in G$.
\end{proof}

\begin{definition}
    Let $(G,\cdot)$ be a group and let $\gamma,\delta\colon G\to\aut(G)$ be two gamma functions on $G$. We say that $\gamma$ and $\delta$ \textbf{mutually normalize each other} if their associated regular subgroups of $\hol(G)$ mutually normalize each other, or equivalently, if $\aut(G)$ is abelian, if $\gamma$ and $\delta$ fulfill the condition \eqref{eqn:neo}.
\end{definition}

In the case where $G$ is a cyclic group, we obtain the following fundamental result.

\begin{corollary} \thlabel{thm:neo3}
    Let $G=\C_m$ be a cyclic group of order $m$, and let $\gamma,\delta\colon G\to\aut(G)$ be two gamma functions on $G$. Suppose that $\gamma$ is defined modulo $q$ and that $\delta$ is defined modulo $r$. Then
    $\gamma$ and $\delta$ mutually normalize each other if and only if
    \[
    \begin{cases}
        x\equiv x^{\delta(y)}+y-y^{\gamma(x)} \pmod{q} \\
        x\equiv x^{\gamma(y)}+y-y^{\delta(x)} \pmod{r} \\
    \end{cases}\quad\forall x,y\in G.
    \]
\end{corollary}

Roughly speaking, we have translated the tough group-theoretical notion of “mutual normalization of regular subgroups" only in terms of a pair of equation in modular arithmetic, which is easier both to be proved or disproved. By using \thref{thm:neo3}, we are ready to solve the mutual normalization problem. Again, we proceed by steps. After some notation, we prove the mutual normalization among pairs of gamma functions associated with regular subgroups belonging to different isomorphism types, trying all the possible combinations. In the end, since we have a complete characterization in \thref{thm:neo3}, it will be easy to conclude that no other mutual normalizations can exist.

\begin{proposition} \thlabel{thm:neo_123456}
    We have
    \begin{enumerate}
        \item $\{\gamma_1,\gamma_2,\gamma_3,\gamma_4\}$ mutually normalize each other.
        \item $\{\gamma_5,\gamma_6\}$ mutually normalize each other.
        \item $\{\gamma_3,\gamma_4,\gamma_5,\gamma_6\}$ mutually normalize each other.
    \end{enumerate}
\end{proposition}

\begin{proof} \mbox{}
    \begin{enumerate} 
        \item Since automorphisms of $G$ are multiplications by odd numbers, we have that $x^\alpha\equiv x\pmod2$ for all $x\in G$ and every $\alpha\in\aut(G)$. Thus, for every pair of gamma functions defined modulo $2$, the equations of \thref{thm:neo3} trivially hold, that is, every pair of gamma functions defined modulo 2 mutually normalize each other. The same computations hold if we substitute any gamma functions defined modulo 2 with the trivial function $\g_1$.
        \item Denote by $\g=\g_5$ and $\d=\g_6$, for the sake of simplicity. Observe that for every $x\in G$
        \begin{equation} \label{eqn:neog56}
            \begin{aligned}
                & x\equiv 0,2\pmod4 \implies y^{\g(x)}\equiv y^{\d(x)}\equiv y\pmod4 \\
                & x\equiv 1,3\pmod4 \implies y^{\g(x)}\equiv y^{\d(x)}\equiv -y\pmod4
            \end{aligned}
        \end{equation}
        Consider the first equation of \thref{thm:neo3} and let $x,y\in G$. If $y$ is even, taking into account \eqref{eqn:neog56}, it turns out to be $$x\equiv x+y-y^{\d(x)}\iff y\equiv y^{\d(x)}\iff y\equiv\pm y\equiv y\pmod4$$
        which is true because $y$ is even. If $y$ is odd, then it becomes $$x\equiv -x +y-y^{\d(x)}\iff 2x\equiv y-y^{\d(x)}\pmod4$$ which is true whether $x$ is even or odd. The second equation can be verified in the same way. Then $\g$ and $\d$ mutually normalize each other.
        \item We already know that $\{\g_3,\g_4\}$ and $\{\g_5,\g_6\}$ mutually normalize each other, respectively. Because of the structure of these gamma functions, we can prove all the remaining mutual normalizations in one shot. Let $\g\in\{\g_3,\g_4\}$ and $\d\in\{\g_5,\g_6\}$. Because of \thref{thm:neo3}, $\g$ and $\d$ mutually normalize each other if and only if 
        \[
        \begin{cases}
            x\equiv x^{\delta(y)}+y-y^{\gamma(x)} \pmod{2} \\
            x\equiv x^{\gamma(y)}+y-y^{\delta(x)} \pmod{4} \\
        \end{cases}\quad\forall x,y\in G.
        \]
        The first equation can be easily verified because for $x,y\in G$
        $$x^{\delta(y)}+y-y^{\gamma(x)}=x+y-y\equiv x \pmod{2}.$$ For the second one, if $y$ is even, it turns out to be $$x\equiv x+y-y^{\d(x)}\iff y\equiv y^{\d(x)}\iff y\equiv\pm y\equiv y\pmod4$$
        which is true because $y$ is even. If $y$ is odd, then it becomes $$x\equiv -x +y-y^{\d(x)}\iff 2x\equiv y-y^{\d(x)}\pmod4$$ which is true whether $x$ is even or odd. Then $\g$ and $\d$ mutually normalize each other.
    \end{enumerate}
\end{proof}

\begin{definition} \thlabel{def:gamma_k}
    Let $\gamma\colon G\to\aut(G)$ be a gamma function on $G$ and let $\sigma_{2k+1}\in\aut(G)$, for some $k\in\{0,\dots,2^{n-1}-1\}$. We denote by $\gamma^k$ the conjugate gamma function of $\gamma$ under $\sigma_{2k+1}\mo\in\aut(G)$ as in \thref{thm:gamma_alpha_lemma}, that is $$\gamma^k =\gamma^{\sigma_{2k+1}\mo}.$$
\end{definition}

\begin{definition}
    We denote as follows some relevant families of gamma functions associated with regular subgroups of $\hol(G)$.
    \[
    \begin{aligned}
        &\Gamma_{\mathsf{p}}=\left\{\gamma_{\mathsf{p}}^k\, :\,0\le k<2^{n-1}\right\}, \\
        &\Gamma_{\mathsf{m}}=\left\{\gamma_{\mathsf{m}}^k\, :\,0\le k<2^{n-1}\right\}, \\
        &\Gamma_{\mathsf{c}}=\left\{\gamma_{\mathsf{c},u}^k\, :\,0\le k<2^{n-1},\ 2\le u\le n\right\}.
    \end{aligned}
    \]
\end{definition}

\begin{lemma} \thlabel{thm:neo_c}
    Two gamma functions $\gamma_{\mathsf{c},u}^k,\gamma_{\mathsf{c},v}^h\in\Gamma_\mathsf{c}$ mutually normalize each other if and only if
    \[
        2^u(2k+1)\equiv 2^v(2h+1)\pmod{2^{n-\max\{u,v\}}}
    \]
\end{lemma}

\begin{proof}
    We know that $\gamma_{\mathsf{c},u}^k$ is defined modulo $2^{n-u}$ and that $\gamma_{\mathsf{c},v}^h$ is defined modulo $2^{n-v}$, therefore from \thref{thm:neo3}, they mutually normalize each other if and only if for every $x,y\in G$
    \begin{align*}
        &\begin{cases}
            x\equiv x^{\gamma_{\mathsf{c},v}^h(y)}+y-y^{\gamma_{\mathsf{c},u}^k(x)} \pmod{2^{n-u}} \\
            x\equiv x^{\gamma_{\mathsf{c},u}^k(y)}+y-y^{\gamma_{\mathsf{c},v}^h(x)} \pmod{2^{n-v}} 
        \end{cases}
        \\[1em]
        &\begin{cases}
            x\equiv x^{\s_{2^v(2h+1)y+1}}+y-y^{\s_{2^u(2k+1)x+1}} \pmod{2^{n-u}} \\
            x\equiv x^{\s_{2^u(2k+1)y+1}}+y-y^{\s_{2^v(2h+1)x+1}} \pmod{2^{n-v}} 
        \end{cases}
        \\[1em]
        &\begin{cases}
            x\equiv (2^v(2h+1)y+1)x+y-(2^u(2k+1)x+1)y \pmod{2^{n-u}} \\
            x\equiv (2^u(2k+1)y+1)x+y-(2^v(2h+1)x+1)y \pmod{2^{n-v}} 
        \end{cases}
        \\[1em]
        &\begin{cases}
            2^u(2k+1)xy\equiv 2^v(2h+1)xy \pmod{2^{n-u}} \\
            2^u(2k+1)xy\equiv 2^v(2h+1)xy \pmod{2^{n-v}} 
        \end{cases}
    \end{align*}
    Observing that these equations are equivalent if we substitute $x=y=1$, and that one implies the other in case $u$ and $v$ were different, the proof is accomplished.

\end{proof}

\begin{proposition} \thlabel{thm:neo_c1}
    The family $$H=\left\{\gamma_{\mathsf{c},u}^k\in\Gamma_\mathsf{c}\, :\, \ceil{\frac{n}{2}}\le u\le n\right\}$$
    is composed of $2^{n-\ceil{\frac{n}{2}}}$ gamma functions, and they mutually normalize each other.
\end{proposition}

\begin{proof}
    The fact that they mutually normalize each other follows directly from \thref{thm:neo_c}. To determine the cardinality of $H$, we know that the conjugacy class of each $\gamma_{\mathsf{c},u}$ contains exactly $2^{n-u-1}$ elements, for every $u\in\{2,\dots,n-1\}$, and that the conjugacy class of $\gamma_{\mathsf{c},n}=\gamma_1$ is a singleton. Therefore
    \begin{equation*}
        |H| = 1+\sum_{u=\ceil{\frac{n}{2}}}^{n-1} 2^{n-u-1} =2^{n-\ceil{\frac{n}{2}}}
    \end{equation*}  
\end{proof}

\begin{proposition} \thlabel{thm:neo_c2}
    Let $\gamma_{\mathsf{c},u}^k,\gamma_{\mathsf{c},v}^h\in\Gamma_\mathsf{c}$ be two gamma functions such that either $2\le v<\ceil{\frac{n}{2}}\le u\le n$ or $2\le v < u <\ceil{\frac{n}{2}}.$ Then $\gamma_{\mathsf{c},u}^k$ and $\gamma_{\mathsf{c},v}^h$ do not mutually normalize each other.
\end{proposition}

\begin{proof}
    Assume the first condition. Observe that, since $u\ge\ceil{\frac{n}{2}}$, we have
    $2^{n-u} \le 2^{n-\ceil{\frac{n}{2}}} \le 2^{\frac{n}{2}} \le 2^{\ceil{\frac{n}{2}}} \le 2^u$,
    then $2^u\equiv 0\pmod{2^{n-u}}$. In the same way, since $v < \ceil{\frac{n}{2}}$ we have $v \le \floor{\frac{n}{2}}$ and
    $2^{n-v} \ge 2^{n-\floor{\frac{n}{2}}} \ge 2^{\frac{n}{2}} \ge 2^{\floor{\frac{n}{2}}} \ge 2^v$,
    that is $2^{n-v}\ge 2^v$, and the equality holds if and only if $v=\frac{n}{2}$ but this is impossible since $v < \ceil{\frac{n}{2}}$. Thus $2^{n-v} > 2^v$ and $2^v\not\equiv 0\pmod{2^{n-v}}$. If we neglect the invertible odd factors, we may rewrite the conditions of \thref{thm:neo_c} equivalently as
    \begin{equation} \label{eqn:neo_c2_2}
        \begin{cases}
            2^v\equiv 0\pmod{2^{n-u}} \\
            2^u\not\equiv 0\pmod{2^{n-v}} 
        \end{cases}
    \end{equation}
    We need to distinguish among two cases. If $u+v\ge n$, then $u\ge n-v$ implies that $2^u\equiv 0\pmod{2^{n-v}}$, in contradiction with \eqref{eqn:neo_c2_2}. Otherwise, if $u+v < n$, then $v < n-u$ implies that $2^v\not\equiv 0\pmod{2^{n-u}}$, again in contradiction with \eqref{eqn:neo_c2_2}. Therefore, the condition of \thref{thm:neo_c} does not hold and then $\gamma_{\mathsf{c},u}^k$ and $\gamma_{\mathsf{c},v}^h$ do not mutually normalize each other. For the second condition is similar.
\end{proof}

\begin{lemma} \thlabel{thm:neo_cu}
    Let $\gamma_{\mathsf{c},u}^k,\gamma_{\mathsf{c},u}^h\in\Gamma_\mathsf{c}$ be two gamma functions such that $2\le u <\ceil{\frac{n}{2}}$. Then $\gamma_{\mathsf{c},u}^k$ and $\gamma_{\mathsf{c},u}^h\in\Gamma_\mathsf{c}$ mutually normalize each other if and only if $$k\equiv h\pmod{2^{n-2u-1}}.$$
\end{lemma}

\begin{proof}
    This is an easy consequence of \thref{thm:neo_c} when $u=v$, indeed $2^u(2k+1)\equiv2^u(2h+1) \pmod{2^{n-u}}$ if and only if $2^{u+1}(k-h)\equiv 0\pmod{2^{n-u}}$, and this holds if and only if $k-h\equiv 0\pmod{2^{n-2u-1}}$.
\end{proof}

\begin{proposition} \thlabel{thm:neo_cu1}
    For every $2\le u < \ceil{\frac{n}{2}}$ and every $0\le t<2^{n-2u-1}$, the family $$A_u^t=\left\{\gamma_{\mathsf{c},u}^k\in\Gamma_\mathsf{c}\, :\, k\equiv t\pmod{2^{n-2u-1}}\right\}$$
    is composed of $2^u$ gamma functions, and they mutually normalize each other. In total, there are $\frac13\left(2^{n-3}-2^{n-2\ceil{\frac{n}{2}}+1}\right)$ distinct $A_u^t$.
\end{proposition}

\begin{proof}
    Two elements of such a family mutually normalize each other because of \thref{thm:neo_cu}. Let us count the elements of $A_u^t$. Observe that, once fixed $2\le u < \ceil{\frac{n}{2}}$, two families $A_{u}^{t_1}, A_{u}^{t_2}$ have the same number of elements because every $\gamma_{\mathsf{c},u}^k$ is defined modulo $2^{n-u}$ and there is a bijection $\varphi\colon A_{u}^{t_1}\to A_{u}^{t_2}$ defined by $\gamma_{\mathsf{c},u}^k \mapsto\gamma_{\mathsf{c},u}^h$, where $k$ and $h$ are such that $k=q(2^{n-2u-1})+t_1$ and $h=q(2^{n-2u-1})+t_2$, for the same $q\in\Z$. Therefore, recalling that the conjugacy class of $\gamma_{\mathsf{c},u}^k$ has $2^{n-u-1}$ different elements, dividing by all the possible choices of $t$, we obtain that $$|A_u^t|=\frac{2^{n-u-1}}{2^{n-2u-1}}=2^u$$ for every $2\le u < \ceil{\frac{n}{2}}$ and $0\le t<2^{n-2u-1}$. Moreover, for every fixed $2\le u < \ceil{\frac{n}{2}}$ there are $2^{n-2u-1}$ distinct $A_u^t$, therefore in total there are
    $$\sum_{u=2}^{\ceil{\frac{n}{2}}-1}\sum_{t=0}^{2^{n-2u-1}-1}1 = \sum_{u=2}^{\ceil{\frac{n}{2}}-1}{2^{n-2u-1}} = \frac13\left(2^{n-3}-2^{n-2\ceil{\frac{n}{2}}+1}\right).$$
\end{proof}

In a very similar manner, we prove the following results.

\begin{lemma} \thlabel{thm:gamma_same_conj}
     For every $0\le k,h < 2^{n-1}$ we have
     \[
     \begin{cases}
         \gamma_\mathsf{p}^k=\gamma_\mathsf{p}^h \\
         \gamma_\mathsf{m}^k=\gamma_\mathsf{m}^h
     \end{cases}
     \mbox{ if and only if }\quad k\equiv h\pmod{2^{n-2}}.
     \]
\end{lemma}

\begin{proof}
    Let $\gamma$ be either $\gamma_\mathsf{p}$ or $\gamma_\mathsf{m}$. We already know that $\gamma$ is defined modulo $2^{n-1}$, then
    \[
    \begin{aligned}
        \gamma^k=\gamma^h &\iff\gamma^k(x)=\gamma^h(x)\quad\forall x\in G \\
        &\iff \gamma((2k+1)x)=\gamma((2h+1)x)\quad\forall x\in G \\
        &\iff (2k+1)x=(2h+1)x\pmod{2^{n-1}}\quad\forall x\in G \\
        &\iff 2kx=2hx\pmod{2^{n-1}}\quad\forall x\in G \\
        &\iff 2k=2h\pmod{2^{n-1}} \\
        &\iff 2(k-h)\pmod{2^{n-1}} \\
        &\iff k-h\equiv 0\pmod{2^{n-2}} \\
    \end{aligned}
    \]
    that is the conclusion.
\end{proof}

\begin{lemma} \thlabel{thm:neo_p}
    Two gamma functions $\gamma_\mathsf{p}^k,\gamma_\mathsf{p}^h\in\Gamma_\mathsf{p}$ mutually normalize each other if and only if
    \[
        k\equiv h\pmod{2^{n-3}}.
    \]
\end{lemma}

\begin{proof}
    Let $\gamma_\mathsf{p}^k,\gamma_\mathsf{p}^h\in\Gamma_\mathsf{p}$. We know that they are both defined modulo $2^{n-1}$, therefore, from \thref{thm:neo3}, they mutually normalize each other if and only if for every $x,y\in G$
    \begin{align*}
        &\begin{cases}
            x\equiv x^{\gamma_{\mathsf{p}}^h(y)}+y-y^{\gamma_{\mathsf{p}}^k(x)} \pmod{2^{n-1}} \\
            x\equiv x^{\gamma_{\mathsf{p}}^k(y)}+y-y^{\gamma_{\mathsf{p}}^h(x)} \pmod{2^{n-1}} 
        \end{cases}
        \\[1em]
        &\begin{cases}
            x\equiv x^{\s_{2(2h+1)y+1}}+y-y^{\s_{2(2k+1)x+1}} \pmod{2^{n-1}} \\
            x\equiv x^{\s_{2(2k+1)y+1}}+y-y^{\s_{2(2h+1)x+1}} \pmod{2^{n-1}} 
        \end{cases}
        \\[1em]
        &\begin{cases}
            x\equiv (2(2h+1)y+1)x+y-(2(2k+1)x+1)y \pmod{2^{n-1}} \\
            x\equiv (2(2k+1)y+1)x+y-(2(2h+1)x+1)y \pmod{2^{n-1}} 
        \end{cases}
        \\[1em]
        &\begin{cases}
            2(2k+1)xy\equiv 2(2h+1)xy \pmod{2^{n-1}} \\
            2(2k+1)xy\equiv 2(2h+1)xy \pmod{2^{n-1}} 
        \end{cases}
    \end{align*}
    Observe that these equations are equivalent if we substitute $x=y=1$, moreover one of them is redundant, thus $$2(2k+1)\equiv 2(2h+1) \pmod{2^{n-1}}\iff 4(k-h)\equiv 0\pmod{2^{n-1}}$$ and this holds if and only if $k-h\equiv 0\pmod{2^{n-3}}.$
\end{proof}

\begin{proposition} \thlabel{thm:neo_p1}
    For every $0\le k < 2^{n-1}$, the family $\{\gamma_\mathsf{p}^k,\gamma_\mathsf{p}^{k+2^{n-3}}\}\subseteq\Gamma_\mathsf{p}$ is composed of $2$ distinct gamma functions, and they mutually normalize each other. 
\end{proposition}

\begin{proof}
    Because of \thref{thm:gamma_same_conj}, we know that $\gamma_\mathsf{p}^k$ and $\gamma_\mathsf{p}^{k+2^{n-3}}$ are distinct and that they fulfill the hypothesis of \thref{thm:neo_p}, so they mutually normalize each other.
\end{proof}

\begin{lemma} \thlabel{thm:neo_m}
    Two gamma functions $\gamma_\mathsf{m}^k,\gamma_\mathsf{m}^h\in\Gamma_\mathsf{m}$ mutually normalize each other if and only if
    \[
        k\equiv h \pmod{2^{n-3}}.
    \]
\end{lemma}

\begin{proof}
    Let $\gamma_\mathsf{m}^k,\gamma_\mathsf{m}^h\in\Gamma_\mathsf{m}$. For the sake of simplicity, denote them by 
    \[
    \begin{aligned}
        \gamma_\mathsf{m}^k(x) &=\sigma_{2(2k+1)x+\varepsilon_x2^{n-2}+1} \\
        \gamma_\mathsf{m}^h(x) &=\sigma_{2(2h+1)x+\varepsilon_x2^{n-2}+1}
    \end{aligned}
    \]
    where 
    \[
    \varepsilon_x=\left\langle
    \begin{array}{lr}
        0 & \quad\mbox{if}\quad x\equiv 0 \pmod2\\
        1 & \quad\mbox{if}\quad x\equiv 1 \pmod2
    \end{array}
    \right.\quad\forall x\in G.
    \]  
    We know that they are both defined modulo $2^{n-1}$, therefore, from \thref{thm:neo3}, they mutually normalize each other if and only if for every $x,y\in G$
    \begin{align*}
        &\begin{cases}
            x\equiv x^{\gamma_{\mathsf{m}}^h(y)}+y-y^{\gamma_{\mathsf{m}}^k(x)} \pmod{2^{n-1}} \\
            x\equiv x^{\gamma_{\mathsf{m}}^k(y)}+y-y^{\gamma_{\mathsf{m}}^h(x)} \pmod{2^{n-1}} 
        \end{cases}
        \\[1em]
        &\begin{cases}
            x\equiv x^{\s_{2(2h+1)y+\varepsilon_y2^{n-2}+1}}+y-y^{\s_{2(2k+1)x+\varepsilon_x2^{n-2}+1}} \pmod{2^{n-1}} \\
            x\equiv x^{\s_{2(2k+1)y+\varepsilon_y2^{n-2}+1}}+y-y^{\s_{2(2h+1)x+\varepsilon_x2^{n-2}+1}} \pmod{2^{n-1}} 
        \end{cases}
        \\[1em]
        &\begin{cases}
            x\equiv (2(2h+1)y+\varepsilon_y2^{n-2}+1)x+y-(2(2k+1)x+\varepsilon_x2^{n-2}+1)y \pmod{2^{n-1}} \\
            x\equiv (2(2k+1)y+\varepsilon_y2^{n-2}+1)x+y-(2(2h+1)y+\varepsilon_x2^{n-2}+1)y \pmod{2^{n-1}} 
        \end{cases}
        \\[1em]
        &\begin{cases}
            2(2k+1)xy+\varepsilon_x2^{n-2}y\equiv 2(2h+1)xy +\varepsilon_y2^{n-2}x\pmod{2^{n-1}} \\
            2(2k+1)xy+\varepsilon_x2^{n-2}y\equiv 2(2h+1)xy+\varepsilon_y2^{n-2}x \pmod{2^{n-1}} 
        \end{cases}
    \end{align*}
    Observe that these equations are equivalent if we substitute $x=y=1$, moreover one of them is redundant, thus we have $\varepsilon_x=\varepsilon_y=1$ and  $$2(2k+1)\equiv 2(2h+1) \pmod{2^{n-1}}\iff 4(k-h)\equiv 0\pmod{2^{n-1}}$$ and this holds if and only if $k-h\equiv 0\pmod{2^{n-3}}$.
\end{proof}

\begin{proposition} \thlabel{thm:neo_m1}
    For every $0\le k < 2^{n-1}$, the family $\{\gamma_\mathsf{m}^k,\gamma_\mathsf{m}^{k+2^{n-3}}\}\subseteq\Gamma_\mathsf{m}$ is composed of $2$ distinct gamma functions, and they mutually normalize each other. 
\end{proposition}

\begin{proof}
    Because of \thref{thm:gamma_same_conj}, we know that $\gamma_\mathsf{m}^k$ and $\gamma_\mathsf{m}^{k+2^{n-3}}$ are distinct and that they fulfill the hypothesis of \thref{thm:neo_m}, so they mutually normalize each other.
\end{proof}

\begin{lemma} \thlabel{thm:neo_pm}
    Two gamma functions $\gamma_\mathsf{p}^k\in\Gamma_\mathsf{p}$ and $\gamma_\mathsf{m}^h\in\Gamma_\mathsf{m}$ mutually normalize each other if and only if
    \[
        k-h\equiv 2^{n-4}\pmod{2^{n-3}}.
    \]
\end{lemma}

\begin{proof}
    Let $\gamma_\mathsf{p}^k\in\Gamma_\mathsf{p}$ and $\gamma_\mathsf{m}^h\in\Gamma_\mathsf{m}$. For the sake of simplicity, denote by 
    \[
    \begin{aligned}
        \gamma_\mathsf{m}^h(x) &=\sigma_{2(2h+1)x+\varepsilon_x2^{n-2}+1}
    \end{aligned}
    \]
    where 
    \[
    \varepsilon_x=\left\langle
    \begin{array}{lr}
        0 & \quad\mbox{if}\quad x\equiv 0 \pmod2\\
        1 & \quad\mbox{if}\quad x\equiv 1 \pmod2
    \end{array}
    \right.\quad\forall x\in G.
    \]  
    We know that they are both defined modulo $2^{n-1}$, therefore, from \thref{thm:neo3}, they mutually normalize each other if and only if for every $x,y\in G$
    \begin{align*}
        &\begin{cases}
            x\equiv x^{\gamma_{\mathsf{m}}^h(y)}+y-y^{\gamma_{\mathsf{p}}^k(x)} \pmod{2^{n-1}} \\
            x\equiv x^{\gamma_{\mathsf{p}}^k(y)}+y-y^{\gamma_{\mathsf{m}}^h(x)} \pmod{2^{n-1}} 
        \end{cases}
        \\[1em]
        &\begin{cases}
            x\equiv x^{\s_{2(2h+1)y+\varepsilon_y2^{n-2}+1}}+y-y^{\s_{2(2k+1)x+1}} \pmod{2^{n-1}} \\
            x\equiv x^{\s_{2(2k+1)y+1}}+y-y^{\s_{2(2h+1)x+\varepsilon_x2^{n-2}+1}} \pmod{2^{n-1}} 
        \end{cases}
        \\[1em]
        &\begin{cases}
            x\equiv (2(2h+1)y+\varepsilon_y2^{n-2}+1)x+y-(2(2k+1)x+1)y \pmod{2^{n-1}} \\
            x\equiv (2(2k+1)x+1)x+y-(2(2h+1)x+\varepsilon_x2^{n-2}+1)y \pmod{2^{n-1}} 
        \end{cases}
        \\[1em]
        &\begin{cases}
            2(2h+1)xy+\varepsilon_x2^{n-2}y\equiv 2(2k+1)xy \pmod{2^{n-1}} \\
            2(2k+1)xy\equiv 2(2h+1)xy+\varepsilon_y2^{n-2}x \pmod{2^{n-1}}
        \end{cases}
    \end{align*}
    Observe that these equations are equivalent if we substitute $x=y=1$, thus we have $\varepsilon_x=\varepsilon_y=1$ and 
    \begin{align*}
        &\begin{cases}
            2(2h+1)+2^{n-2}\equiv 2(2k+1) \pmod{2^{n-1}} \\
            2(2k+1)\equiv 2(2h+1)+2^{n-2} \pmod{2^{n-1}} 
        \end{cases}
        \\[1em]
        &\begin{cases}
            4(h-k)\equiv -2^{n-2} \pmod{2^{n-1}} \\
            4(k-h)\equiv 2^{n-2} \pmod{2^{n-1}} \\
        \end{cases}
    \end{align*}
    Observe that one equation is redundant, and other holds if and only if for some $t\in\Z$ $$4(k-h)-2^{n-2}=t\cdot 2^{n-1}\iff k-h-2^{n-4}=t\cdot 2^{n-3}$$ that is, if and only if $k-h\equiv 2^{n-4}\pmod{2^{n-3}}$.
\end{proof}

\begin{proposition} \thlabel{thm:neo_pm1}
    For every $0\le k < 2^{n-1}$, the family
    \[
        S_k=\left\{\gamma_\mathsf{p}^{k},\gamma_\mathsf{m}^{k+2^{n-4}},\gamma_\mathsf{p}^{k+2^{n-3}}, \gamma_\mathsf{m}^{k+2^{n-3}+2^{n-4}}\right\}\subseteq \Gamma_{\mathsf{p}}\cup\Gamma_{\mathsf{m}}
     \]
     is composed of $4$ distinct gamma functions, and they mutually normalize each other. In total, there are $2^{n-3}$ distinct $S_k$.
\end{proposition}

\begin{proof}
    Fix $0\le k < 2^{n-1}$. The elements of $S_k$ mutually normalize each other because of \thref{thm:neo_p}, \thref{thm:neo_m}, and \thref{thm:neo_pm}. Moreover, all the elements of $S_k$ are distinct thanks to \thref{thm:gamma_same_conj}. Again, from \thref{thm:gamma_same_conj}, we notice that $S_k=S_h$ if and only if $k\equiv h\pmod{2^{n-3}}$, namely, there are only $2^{n-3}$ choices of the parameter $k$ producing mutually different conjugates, so that there are only $2^{n-3}$ distinct families of the form of $S_k$.
\end{proof}

\begin{proposition} \thlabel{thm:neo_n}
    There are no other mutual normalizations except the ones highlighted above.
\end{proposition}

\begin{proof}
    In the previous propositions, we have taken into account all possible cases, therefore there are no other possibilities except the ones studied so far.
\end{proof}

\begin{theorem}[Local normalizing graph of $\C_{2^n}$] \thlabel{thm:c2ngraph}
    In the notation of the previous propositions, the local normalizing graph of $\C_{2^n}$ for $n\ge 3$ contains the following connected components.
    \begin{enumerate}
        \item The connected component containing the clique $H$ and the subgroups $N_1,\dots,N_6$, composed by $2^{n-\ceil{\frac{n}{2}}}+4$ regular subgroups.
        \item $\frac13\left(2^{n-3}-2^{n-2\ceil{\frac{n}{2}}+1}\right)$ connected components $A_u^t$, each of which is a clique composed by $2^u$ regular subgroups, for $2\le u < \ceil{\frac{n}{2}}$ and $0\le t<2^{n-2u-1}$.
        \item $2^{n-3}$ connected components $S_k$, each of which is a clique composed by $4$ regular subgroups, for $0\le k < 2^{n-3}$.
    \end{enumerate}
    Thus, for some indices $k_1,k_2,k_3,\dots$ and $t_1,t_2,t_3,\dots$ representing conjugation under automorphisms as in the previous statements, the local normalizing graph of $\C_{2^n}$ for $n\ge 3$ has the following form.

    
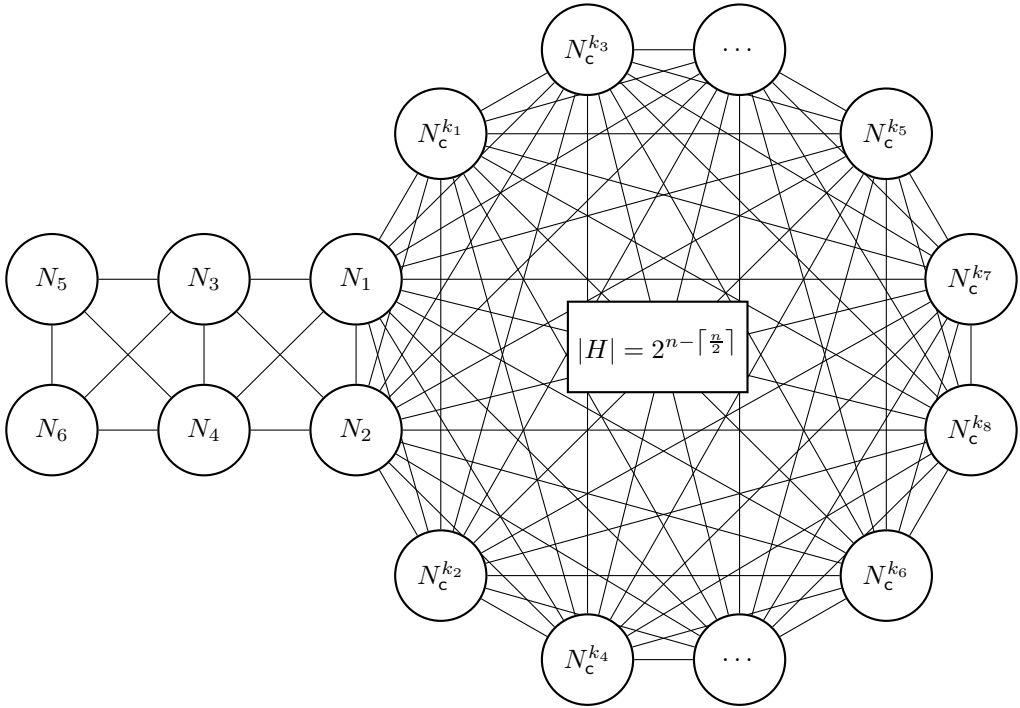
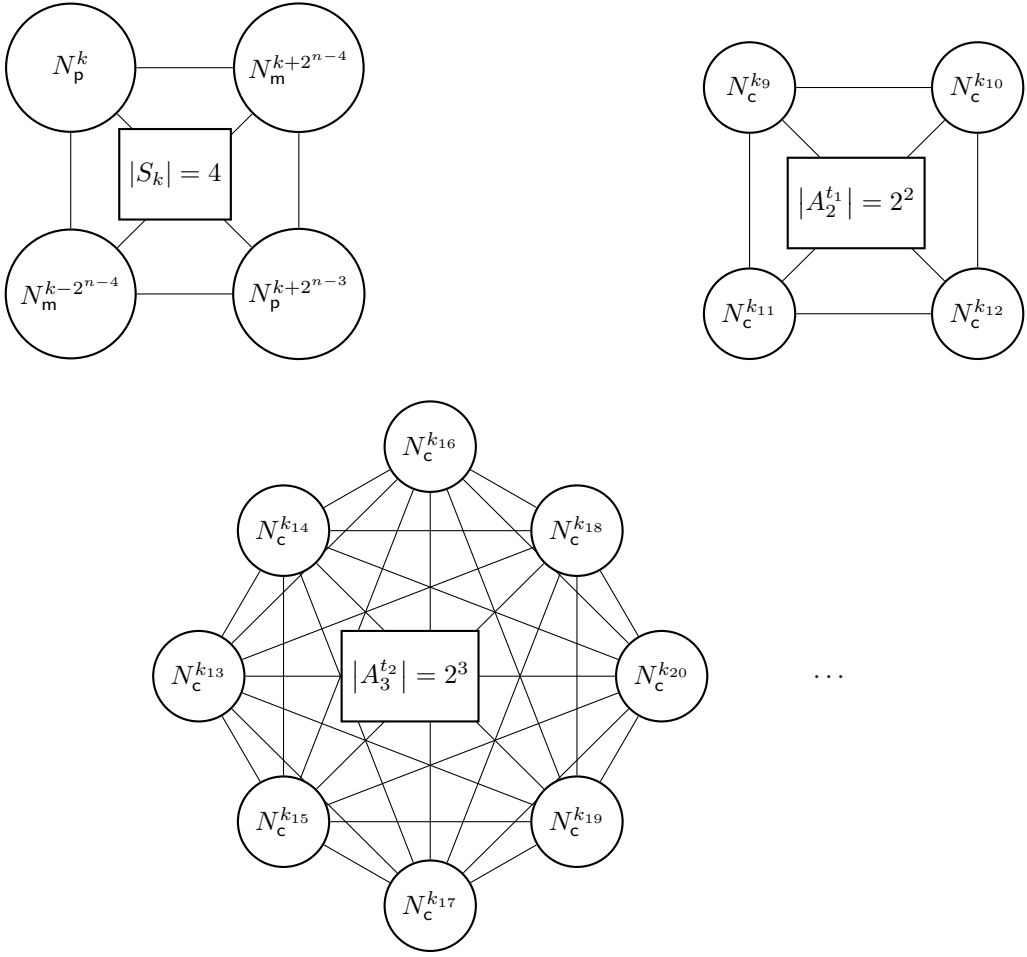
\begin{center}
        \begin{tikzpicture}[
            minimum size=\nodeSize,
            node distance=\nodeDist,
            thick, 
            main/.style = {draw=black, fill=white,  circle},
            ] 
                         
            \node[main] (6) {$N_6$}; 
            \node[main] (5) [above of=6] {$N_5$}; 
            \node[main] (4) [right of=6] {$N_4$}; 
            \node[main] (3) [above of=4] {$N_3$}; 
            \node[main] (2) [right of=4, label={[label distance=2.20cm, fill=white, draw=black]10:$|H|=2^{n-\ceil{\frac{n}{2}}}$}] {$N_2$}; 
            \node[main] (1) [above of=2] {$N_1$};
            
            \node[main] (7) [position=60:{0.5\nodeDist} from 1] {$N_\mathsf{c}^{k_1}$}; 
            \node[main] (8) [position=-60:{0.5\nodeDist} from 2] {$N_\mathsf{c}^{k_2}$}; 
            \node[main] (9) [position=30:{0.5\nodeDist} from 7] {$N_\mathsf{c}^{k_3}$}; 
            \node[main] (10) [position=-30:{0.5\nodeDist} from 8] {$N_\mathsf{c}^{k_4}$}; 
            \node[main] (11) [right of=9] {$\dots$}; 
            \node[main] (12) [right of=10] {$\dots$}; 
            \node[main] (13) [position=-30:{0.5\nodeDist} from 11] {$N_\mathsf{c}^{k_5}$}; 
            \node[main] (14) [position=30:{0.5\nodeDist} from 12] {$N_\mathsf{c}^{k_6}$};
            \node[main] (15) [position=-60:{0.5\nodeDist} from 13] {$N_\mathsf{c}^{k_7}$}; 
            \node[main] (16) [position=60:{0.5\nodeDist} from 14] {$N_\mathsf{c}^{k_8}$};

            \begin{scope}[on background layer]
                \draw (6) -- (5); 
                \draw (6) -- (4); 
                \draw (6) -- (3); 
                \draw (5) -- (4); 
                \draw (5) -- (3);  
                \draw (4) -- (3); 
                \draw (4) -- (2); 
                \draw (4) -- (1); 
                \draw (3) -- (2); 
                \draw (3) -- (1); 
                \draw (2) -- (1); 
                
                \draw (1) -- (7); 
                \draw (1) -- (8); 
                \draw (1) -- (9); 
                \draw (1) -- (10); 
                \draw (1) -- (11); 
                \draw (1) -- (12); 
                \draw (1) -- (13); 
                \draw (1) -- (14); 
                \draw (1) -- (15); 
                \draw (1) -- (16); 
                
                \draw (2) -- (7);
                \draw (2) -- (8); 
                \draw (2) -- (9); 
                \draw (2) -- (10); 
                \draw (2) -- (11); 
                \draw (2) -- (12); 
                \draw (2) -- (13); 
                \draw (2) -- (14); 
                \draw (2) -- (15); 
                \draw (2) -- (16); 
                
                \draw (7) -- (8);
                \draw (7) -- (9); 
                \draw (7) -- (10); 
                \draw (7) -- (11); 
                \draw (7) -- (12); 
                \draw (7) -- (13); 
                \draw (7) -- (14); 
                \draw (7) -- (15); 
                \draw (7) -- (16); 
                
                \draw (8) -- (9); 
                \draw (8) -- (10); 
                \draw (8) -- (11); 
                \draw (8) -- (12); 
                \draw (8) -- (13); 
                \draw (8) -- (14); 
                \draw (8) -- (15); 
                \draw (8) -- (16); 
                
                \draw (9) -- (10); 
                \draw (9) -- (11); 
                \draw (9) -- (12); 
                \draw (9) -- (13); 
                \draw (9) -- (14); 
                \draw (9) -- (15); 
                \draw (9) -- (16); 
                
                \draw (10) -- (11); 
                \draw (10) -- (12); 
                \draw (10) -- (13); 
                \draw (10) -- (14); 
                \draw (10) -- (15); 
                \draw (10) -- (16); 
                
                \draw (11) -- (12); 
                \draw (11) -- (13); 
                \draw (11) -- (14); 
                \draw (11) -- (15); 
                \draw (11) -- (16); 
                
                \draw (12) -- (13); 
                \draw (12) -- (14); 
                \draw (12) -- (15); 
                \draw (12) -- (16); 
                
                \draw (13) -- (14); 
                \draw (13) -- (15); 
                \draw (13) -- (16); 
                
                \draw (14) -- (15); 
                \draw (14) -- (16); 
                
                \draw (15) -- (16); 
            \end{scope}
        \end{tikzpicture}
        \\[5mm]
        \begin{tikzpicture}[
            minimum size=\nodeSize,
            node distance=30mm,
            thick, 
            main/.style = {draw=black, fill=white,  circle},
            ]              
     
            \node[main] (1) [minimum size=17mm, label={[label distance=1.5mm, fill=white, draw=black]-52:$|S_k|=4$}] {$N_\mathsf{p}^k$}; 
            \node[main] (2) [right of=1] {$N_\mathsf{m}^{k+2^{n-4}}$};     
            \node[main] (3) [below of=1] {$N_\mathsf{m}^{k-2^{n-4}}$}; 
            \node[main] (4) [below of=2] {$N_\mathsf{p}^{k+2^{n-3}}$}; 
            
            \begin{scope}[on background layer]
                \draw (1) -- (2); 
                \draw (1) -- (3); 
                \draw (1) -- (4); 
                \draw (2) -- (3); 
                \draw (2) -- (4);  
                \draw (3) -- (4); 
            \end{scope}
        \end{tikzpicture} 
        \qquad\qquad\qquad\qquad\qquad\qquad
        \begin{tikzpicture}[
            minimum size=\nodeSize,
            node distance=30mm,
            thick, 
            main/.style = {draw=black, fill=white,  circle},
            ]              
    
            \node[main] (1) [label={[label distance=4.25mm, fill=white, draw=black, ]-62.5:$\left|A_{2}^{t_1}\right|=2^2$}]{$N_\mathsf{c}^{k_9}$}; 
            \node[main] (2) [right of=1] {$N_\mathsf{c}^{k_{10}}$};     
            \node[main] (3) [below of=1] {$N_\mathsf{c}^{k_{11}}$}; 
            \node[main] (4) [below of=2] {$N_\mathsf{c}^{k_{12}}$}; 
            
            \begin{scope}[on background layer]
                \draw (1) -- (2); 
                \draw (1) -- (3); 
                \draw (1) -- (4); 
                \draw (2) -- (3); 
                \draw (2) -- (4);  
                \draw (3) -- (4); 
            \end{scope}
        \end{tikzpicture}
        \\[5mm]
        \begin{tikzpicture}[
            minimum size=\nodeSize,
            node distance=\nodeDist,
            thick, 
            main/.style = {draw=black, fill=white,  circle},
            ] 
                         
            \node[main] (1) [label={[label distance=1.25cm, fill=white, draw=black]0:$\left|A_{3}^{t_2}\right|=2^3$}] {$N_\mathsf{c}^{k_{13}}$}; 
            \node[main] (2) [position=60:{0.5\nodeDist} from 1] {$N_\mathsf{c}^{k_{14}}$}; 
            \node[main] (3) [position=-60:{0.5\nodeDist} from 1] {$N_\mathsf{c}^{k_{15}}$}; 
            \node[main] (4) [position=30:{0.5\nodeDist} from 2] {$N_\mathsf{c}^{k_{16}}$}; 
            \node[main] (5) [position=-30:{0.5\nodeDist} from 3] {$N_\mathsf{c}^{k_{17}}$}; 
            \node[main] (6) [position=-30:{0.5\nodeDist} from 4] {$N_\mathsf{c}^{k_{18}}$};
            \node[main] (7) [position=30:{0.5\nodeDist} from 5] {$N_\mathsf{c}^{k_{19}}$};
            \node[main] (8) [position=-60:{0.5\nodeDist} from 6, label={[label distance=10mm]0:$\dots$}] {$N_\mathsf{c}^{k_{20}}$};
            
            \begin{scope}[on background layer]
                \draw (1) -- (2); 
                \draw (1) -- (3); 
                \draw (1) -- (4); 
                \draw (1) -- (5); 
                \draw (1) -- (6); 
                \draw (1) -- (7);
                \draw (1) -- (8);
                
                \draw (2) -- (3); 
                \draw (2) -- (4); 
                \draw (2) -- (5); 
                \draw (2) -- (6); 
                \draw (2) -- (7);
                \draw (2) -- (8);
                
                \draw (3) -- (4); 
                \draw (3) -- (5); 
                \draw (3) -- (6); 
                \draw (3) -- (7);
                \draw (3) -- (8);
                
                \draw (4) -- (5); 
                \draw (4) -- (6); 
                \draw (4) -- (7);
                \draw (4) -- (8);
                
                \draw (5) -- (6); 
                \draw (5) -- (7);
                \draw (5) -- (8);
                
                \draw (6) -- (7);
                \draw (6) -- (8);
                
                \draw (7) -- (8);
            \end{scope}
        \end{tikzpicture}   
    \end{center}
\end{theorem}

\section{The classification in the case \texorpdfstring{$p$}{p} is odd} \label{sec:podd}

In this section, we generalize the previous results when $p$ is an odd prime. Unlike the case $p=2$, the odd case is more straightforward, and we can deal with it without any distinctions among small cases, general cases and the exponent of $p$. Most of the proof will be very similar to the case $p=2$, in some cases they are obtained just by substituting the symbol $2$ with $p$. We report only the most substantial difference of them, despite the general approach is the same as above. We denote by $p$ an odd prime and by $G=\C_{p^n}$ a cyclic group of order $p^n$ written in additive notation, where $n\ge 1$, unless otherwise stated. 

\subsection{Existence problem}

The proofs of the results in this subsection are substantially identical to those of the case $p=2$. 

\begin{proposition} \thlabel{thm:cpng1}
    The map $\gamma_n\colon G \to\aut(G)$ defined by
    \[
    \gamma_n(x)=\sigma_1\qquad\forall x\in G
    \]    
    is a gamma function on $G$, and the associated regular subgroup $N_n\le\hol(G)$ is isomorphic to $\C_{p^n}$. Moreover, $N_n$ is normal in $\hol(G)$.
\end{proposition}

\begin{proof}
    This proof is identical to that in the case $p=2$.
\end{proof}

As above we need an arithmetic lemma to conclude the existence problem.

\begin{lemma} \thlabel{thm:puk_lemma_mod}
    For every $n,k,u\in\N$ such that $n\ge 1,\ 1\le u< n$ and $1\le k \le p^n$ we have that $$p^{-u}\left[(p^u+1)^k-1\right]\equiv 0\pmod{p^n}\iff k=p^n.$$
\end{lemma}

\begin{proposition} \thlabel{thm:cpngc}
    The map $\gamma_{u}\colon G \to\aut(G)$ defined by
    \[
    \gamma_{u}(x)=\sigma_{p^ux+1}\qquad\forall x\in G
    \]    
    is a gamma function on $G$ for every $u=1,\dots,n$, and the associated regular subgroup $N_{u}\le\hol(G)$ is isomorphic to $\C_{p^n}$.
\end{proposition}

\begin{proof}
    Fix $u\in\{1,\dots,n\}$. We may assume that $u\neq n$ because this value of $u$ have been already studied in \thref{thm:cpng1}. Denote by $\g=\g_{u}$, for the sake of simplicity, and observe that $\g$ is defined modulo $p^{n-u}$, indeed for every $x\in G$ with $0\le x<p^{n-u}$, and every $t\in\Z$
    $$\g(x+tp^{n-u})=\s_{p^u(x+tp^{n-u})+1}=\s_{p^ux+tp^n+1}=\s_{p^ux+1}=\g(x).$$
    We divide the proof in two steps.
    
    \begin{description}[font=\normalfont\textit, leftmargin=\parindent]
    
        \item[Step 1.] Let us verify the gamma functional equation for $\g$. For every $x,y\in G$ such that $0\le x,y<p^{n-u}$
        \[
        \begin{aligned}
            \g(x^{\g(y)}+y)&=\g(x^{\s_{p^uy+1}}+y) \\
            &=\g((p^uy+1)x+y) =\g(p^uxy+x+y) \\
            &=\s_{p^u(p^uxy+x+y)+1} =\s_{p^{2u}xy+p^ux+p^uy+1} \\
            &=\s_{(p^ux+1)(p^uy+1)} =\s_{p^ux+1}\s_{p^uy+1} \\
            &=\g(x)\g(y) 
        \end{aligned}
        \]
        Then $\g$ is a gamma function on $G$.
        
        \item[Step 2.] Denote by $\c$ the circle operation on $G$ induced by $\g$. From \thref{thm:srg1}, we know that the map $\nu=\nu_{u}\colon(G,\c)\to N_{u}$ is an isomorphism of groups, hence, it is enough to show that $(G,\c)\cong\C_{p^n}$. We are looking for an element $r\in G$ such that
        $$(G,\c)=\left\langle r\, :\,r^{\c p^{n}}=0\right\rangle\cong\C_{p^n}.$$
        We claim that
        \begin{equation} \label{eqn:cpngcu_1}
            1^{\c k}=\sum_{i=0}^{k-1}\binom{k}{i}(p^u)^{k-1-i}\qquad\forall k\in\N
        \end{equation}
        and prove it by induction. Trivially $1^{\c 0}=0$ (because the sum in \eqref{eqn:cpngcu_1} is empty) and $1^{\c 1}=1$. Assume \eqref{eqn:cpngcu_1} true for some $k\in\N$. Then
        \begin{align*}
            1^{\c(k+1)} &= 1^{\c k}\c 1 = \left(1^{\c k}\right)^{\g(1)}+1 = \left(p^u+1\right)\cdot\left(\sum_{i=0}^{k-1}\binom{k}{i}(p^u)^{k-1-i}\right)+1 \\
            &= \sum_{i=0}^{k-1}\binom{k}{i}(p^u)^{k-1-(i-1)} + \sum_{i=0}^{k-1}\binom{k}{i}(p^u)^{k-1-i} +1 \\
            &= \sum_{j=-1}^{k-2}\binom{k}{j+1}(p^u)^{k-1-j} + \sum_{j=0}^{k-1}\binom{k}{j}(p^u)^{k-1-j} +1 \\
            &= \binom{k}{0}(p^u)^{k} + \sum_{j=0}^{k-2}\left[\binom{k}{j+1}+\binom{k}{j}\right](p^u)^{k-1-j} + \binom{k}{k-1}(p^u)^0 +1 \\
            &= \binom{k+1}{0}(p^u)^{(k+1)-1} + \sum_{j=0}^{k-2}\binom{k+1}{j+1}(p^u)^{k-1-j} + \binom{k+1}{k}(p^u)^0 \\
            &= \sum_{j=-1}^{k-1}\binom{k+1}{j+1}(p^u)^{k-1-j} = \sum_{i=0}^{k}\binom{k+1}{i}(p^u)^{k-i} \\
            &= \sum_{i=0}^{(k+1)-1}\binom{k+1}{i}(p^u)^{(k+1)-1-i}
        \end{align*}
        where we used the well known facts $$\binom{k}{j+1}+\binom{k}{j}=\binom{k+1}{j+1},\qquad\qquad\binom{k}{0}=\binom{k+1}{0}.$$
        Moreover we can simplify this expression as follows: for every $k\in\N$
        \[
        \begin{aligned}
            1^{\c k} &=\sum_{i=0}^{k-1}\binom{k}{i}(p^u)^{k-1-i}=p^{-u}\sum_{i=0}^{k-1}\binom{k}{i}(p^u)^{k-i} \\
            &= p^{-u}\left[\sum_{i=0}^{k}\binom{k}{i}(p^u)^{k-i} -1\right] = p^{-u}\left[(p^u+1)^k-1\right]\\
        \end{aligned}
        \]
        and, from the \thref{thm:puk_lemma_mod}, we know that for $0\le k\le p^n$ we have
        \begin{equation} \label{eqn:cpngcu_2}
            1^{\c k}=p^{-u}\left[(p^u+1)^k-1\right]\equiv 0\pmod{p^n}\iff k=p^n
        \end{equation}
        that is, the order of $1\in G$ with respect to the circle operation $\c$ is $p^n$. We can set $r=1$ and, thanks to \eqref{eqn:cpngcu_2}, it is the sought generator, so we can conclude that $(G,\c)$ is isomorphic to $\C_{p^n}$.  \qedhere
    \end{description}
\end{proof}

\begin{lemma} \thlabel{thm:cpn_k_lemma}
    Let $G=\C_{p^n}$ be a cyclic group of order $p^n$ and let $N\le\hol(G)$ be a regular subgroup. Then $|\nor_{\sym(G)}(N)\cap\aut(G)|=|K|$ where $$K=\left\{(k,c)\in\Z\times\Z\, :\,0\le k<p^{n-1},\ 1\le c < p,\ \g(x)=\g((kp+c)x)\quad \forall x\in G\right\}.$$
\end{lemma}

\begin{proof}
    To compute the number $|\nor_{\sym(G)}(N)\cap\aut(G)|$ it is enough to find how many $\alpha\in\aut(G)$ are such that $\gamma^\alpha=\gamma$, that is, since $\aut(G)$ is abelian, when 
    \begin{align*}
        \g^\alpha(x)=\g(x)\quad\forall x\in G &\iff\alpha\mo\g(x^{\alpha\mo})\alpha=\g(x)\quad\forall x\in G\\
        &\iff \g(x^{\alpha\mo})=\g(x)\quad\forall x\in G\\
        &\iff \g(x)=\g(x^{\alpha})\quad\forall x\in G  
    \end{align*}
    where the last equivalence holds thanks to the substitution $x\mapsto x^\alpha$ and the arbitrariness of $x\in G$. Moreover, each element $\alpha\in\aut(G)$ is of the form $\alpha=\sigma_{kp+c}$ for some $0\le k<p^{n-1}$ and $\ 1\le c < p$, therefore, it is enough to find how many pairs of $(k,c)$ are such that $$\g(x)=\g((kp+c)x)\qquad\forall x\in G$$ so $|\nor_{\sym(G)}(N)\cap\aut(G)|=|K|$.
\end{proof}

\begin{proposition} \thlabel{thm:cpnclassgc}
    There are (disjoint) conjugacy classes of regular subgroups isomorphic to $\C_{p^{n}}$ of sizes $1,p-1,p^2-p,p^3-p^2,\dots,p^{n-1}-p^{n-2}$, namely they are $$\left\{N_{u}^\alpha\, :\,\alpha\in\aut(G)\right\}\qquad u=1,\dots,n$$ of size $p^{n-u}-p^{n-u-1}$, for $u\neq n$.
\end{proposition}

\begin{proof}
    Fix $u\in\{1,\dots,n\}$. We may assume that $u\neq n$ because this value of $u$ have been already studied in \thref{thm:cpng1}, and, since $N_n\norm\hol(G)$, its conjugacy class is a singleton. Let us reconsider the gamma function $\g_{u}\colon G\to\aut(G)$ of \thref{thm:cpngc} and denote it by $\g$, for the sake of simplicity. From \thref{thm:cpn_k_lemma}, we are looking for pairs $(k,c)\in\Z\times\Z$ are such that $0\le k<p^{n-1}$, $\ 1\le c < p$, and $\g(x)=\g((kp+c)x)$ for every $x\in G$, that is, since $\g$ is defined modulo $p^{n-u}$,    
    $$(kp+c)x\equiv x\pmod{p^{n-u}}.$$
    In particular, the previous condition needs to be fulfilled for $x=1$, therefore it is equivalent to $kp+(c-1)\equiv 0\pmod{p^{n-u}}$, that is, because $1\le c <p$ and $kp$ is a multiple of $p$,
    \[
    \begin{cases}
        c=1 \\
        k=t\cdot p^{n-(u+1)}
    \end{cases}
    \]
    for some $0\le t<p^u$. Thus, $$K=\left\{(t\cdot p^{n-u-1},1)\in\Z\times\Z\, :\,0\le t < p^u \right\}$$
    has cardinality $|K|=p^u$ and, thanks to \thref{thm:cpn_k_lemma} and \thref{thm:gamma_counts}, $$|\Gamma|=\frac{|\!\aut(G)|}{|K|}=\frac{p^{n-1}(p-1)}{p^u}=p^{n-u-1}\cdot(p-1)=p^{n-u}-p^{n-u-1}$$ is the size of the conjugacy class of $\gamma$.
\end{proof}

\begin{corollary}
    There are at least $p^{n-1}$ regular subgroups of $\hol(G)$, and they are all isomorphic to $\C_{p^n}$.
\end{corollary}

\begin{proof}
    It follows by summing together the sizes of the (disjoint) conjugacy classes found above. Taking into account that $\{N_n\}$ is the conjugacy class of $N_n$, the total number of those conjugates is $$1+\sum_{u=1}^{n-1}\left(p^{n-u}-p^{n-u-1}\right)=p^{n-1}$$
    where the last sum is telescopic. 
\end{proof}

\subsection{Uniqueness problem}

The \emph{uniqueness problem} in the case $p$ odd is way easier than the even case, because it is enough to mention and translate in our notation two results which already exist in literature. We start with a theorem from T. Kohl (see \cite {kohl1998prime} and \cite{byott1996hopfgalois}), which allows us to deal with the cyclic isomorphism type. 

\begin{theorem} [\cite{kohl1998prime}]
    There are exactly $p^{n-1}$ regular subgroups of $\hol(G)$ isomorphic to $\C_{p^n}$.
\end{theorem}

So far, we have found the exact number of cyclic regular subgroups of $\hol(G)$, but, a priori, there could exist also some other regular subgroup of another isomorphism type. This is, in fact, impossible, and we are going to prove it exploiting a result of E. Campedel, A. Caranti, and I. Del Corso. We first state such result and then we use it to reach the conclusion.

\begin{lemma}[\cite{Campedel_2020}] \thlabel{thm:cpn_uniqueness_cyclic}
    Let $G$ be a finite group and let $A\le G$ be a cyclic subgroup of order $p^n$, where $p$ is an odd prime. Let $\gamma\colon A\to\aut(G)$ be a relative gamma function on $A$, and denote by $\circ$ the induced circle operation on $A$. Then, also $(A,\circ)$ is cyclic of order $p^n$.
\end{lemma}

\begin{proposition}
   Each regular subgroup of $\hol(G)$ is cyclic.
\end{proposition}

\begin{proof}
    It follows directly from \thref{thm:cpn_uniqueness_cyclic} by considering $A=G$, which is cyclic of order $p^n$.
\end{proof}

\subsection{Mutual normalization problem}

In this subsection, we exploit several times the ring structure of $\Zmod{p^n}$, in particular the fact that all the elements divisible by $p$ are zero-divisors and all elements of $\Zmod{p^n}$ coprime with $p$ are invertible. 

\begin{definition} \thlabel{def:gamma_kc}
    Let $\gamma\colon G\to\aut(G)$ be a gamma function on $G$ and let $\sigma_{kp+c}\in\aut(G)$, for some $0\le k < p^{n-1}$ and $1\le c <p$. We denote by $\gamma^{k,c}$ the conjugate gamma function of $\gamma$ under $\sigma_{kp+c}\mo\in\aut(G)$ as in \thref{thm:gamma_alpha_lemma}, that is $$\gamma^{k,c} =\gamma^{\sigma_{kp+c}\mo}.$$
\end{definition}

\begin{definition}
    We denote as follows some relevant conjugacy classes of gamma functions associated with regular subgroups of $\hol(G)$, and their union.
    \[
    \begin{aligned}
        &\Gamma_{u}=\left\{\gamma_{u}^{k,c}\, :\,0\le k < p^{n-1},\ 1\le c <p\right\}\qquad u\in\{1,\dots, n\} \\
        &\Gamma =\bigcup_{u=1}^n\Gamma_u=\left\{\gamma_{u}^{k,c}\, :\,0\le k < p^{n-1},\ 1\le c <p,\ 1\le u\le n\right\}
    \end{aligned}
    \]
\end{definition}

\begin{proposition} \thlabel{thm:cpn_neo_c}
    Two gamma functions $\gamma_{u}^{k,c},\gamma_{v}^{h,d}\in\Gamma$ mutually normalize each other if and only if
    \[
        p^u(kp+c)\equiv p^v(hp+d)\pmod{p^{n-\max\{u,v\}}} 
    \]
\end{proposition}

\begin{proof}
    We know that $\gamma_{u}^{k,c}$ is defined modulo $p^{n-u}$ and that $\gamma_{v}^{h,d}$ is defined modulo $p^{n-v}$, therefore from \thref{thm:neo3}, they mutually normalize each other if and only if for every $x,y\in G$
    \begin{align*}
        &\begin{cases}
            x\equiv x^{\gamma_{v}^{h,d}(y)}+y-y^{\gamma_{u}^{k,c}(x)} \pmod{p^{n-u}} \\
            x\equiv x^{\gamma_{u}^{k,c}(y)}+y-y^{\gamma_{v}^{h,d}(x)} \pmod{p^{n-v}} 
        \end{cases}
        \\[1em]
        &\begin{cases}
            x\equiv x^{\s_{p^v(hp+d)y+1}}+y-y^{\s_{p^u(kp+c)x+1}} \pmod{p^{n-u}} \\
            x\equiv x^{\s_{p^u(kp+c)y+1}}+y-y^{\s_{p^v(hp+d)x+1}} \pmod{p^{n-v}} 
        \end{cases}
        \\[1em]
        &\begin{cases}
            x\equiv (p^v(hp+d)y+1)x+y-(p^u(kp+c)x+1)y \pmod{p^{n-u}} \\
            x\equiv (p^u(kp+c)y+1)x+y-(p^v(hp+d)x+1)y \pmod{p^{n-v}} 
        \end{cases}
        \\[1em]
        &\begin{cases}
            p^u(kp+c)xy\equiv p^v(hp+d)xy \pmod{p^{n-u}} \\
            p^u(kp+c)xy\equiv p^v(hp+d)xy \pmod{p^{n-v}} 
        \end{cases}
    \end{align*}
    The last condition must hold for every $x,y\in G$, so in particular, for $x=y=1$. Observing that this particular case is also sufficient for its validity for every $x,y\in G$, and that one congruence implies the other in case $u$ and $v$ were different, the proof is accomplished.
\end{proof}

\begin{proposition} \thlabel{thm:cpn_neo_c1}
    The family $$H=\left\{\gamma_{u}^{k,c}\in\Gamma\, :\, \ceil{\frac{n}{2}}\le u\le n\right\}$$
    is composed of $p^{n-\ceil{\frac{n}{2}}}$ gamma functions, and they mutually normalize each other.
\end{proposition}

\begin{proof}
    We divide the proof in two steps.
    \begin{description}[font=\normalfont\textit, leftmargin=\parindent]
    
        \item[Step 1.] Let $\gamma_{u}^{k,c},\gamma_{v}^{h,d}\in H$. Observe that since $u,v\ge\ceil{\frac{n}{2}}$, we have
        \begin{align*}
            \begin{cases}
                n-u\le n-\ceil{\frac{n}{2}} \\
                n-v\le n-\ceil{\frac{n}{2}}
            \end{cases}
            \implies 
            \begin{cases}
                p^{n-u},p^{n-v}\le p^{n-\ceil{\frac{n}{2}}}\le p^{\ceil{\frac{n}{2}}}\le p^u \\
                p^{n-u},p^{n-v}\le p^{n-\ceil{\frac{n}{2}}}\le p^{\ceil{\frac{n}{2}}}\le p^v
            \end{cases}
        \end{align*}
        and this implies that $p^u$ and $p^v$ are both zero $\bmod{\,p^{n-u}}$ and $\bmod{\,p^{n-v}}$. Therefore the equations 
        \[
        \begin{cases}
            p^u(kp+c)\equiv p^v(hp+d)\pmod{p^{n-u}} \\
            p^u(kp+c)\equiv p^v(hp+d)\pmod{p^{n-v}} 
        \end{cases}
        \]
        hold and from \thref{thm:cpn_neo_c}, $\gamma_{u}^{k,c}$ and $\gamma_{v}^{h,d}$ mutually normalize each other.
        
        \item[Step 2.] Let us count the elements of $H$. We know from \thref{thm:cpnclassgc} that the conjugacy class of each $\gamma_{u}^{k,c}$ contains exactly $p^{n-u}-p^{n-u-1}$ elements, for every $u\in\{1,\dots,n-1\}$, and from \thref{thm:cpng1} that the conjugacy class of $\gamma_{n}^{k,c}=\gamma_1$ is a singleton. Therefore
        \[
        \begin{aligned}
            |H| &= 1+\sum_{u=\ceil{\frac{n}{2}}}^{n-1}\left( p^{n-u}-p^{n-u-1}\right)=p^{n-\ceil{\frac{n}{2}}}
        \end{aligned}
        \]
        where the last sum is telescopic. \qedhere
    \end{description}
\end{proof}

\begin{proposition} \thlabel{thm:cpn_neo_c2}
    Let $\gamma_{u}^{k,c},\gamma_{v}^{h,d}\in\Gamma$ be two gamma functions such that $$1\le v<\ceil{\frac{n}{2}}\le u\le n.$$ Then $\gamma_{u}^{k,c}$ and $\gamma_{v}^{h,d}$ do not mutually normalize each other.
\end{proposition}

\begin{proof}
    From \thref{thm:cpn_neo_c}, $\gamma_{u}^{k,c},\gamma_{v}^{h,d}\in\Gamma$ mutually normalize each other if and only if
    \begin{equation} \label{eqn:cpn_neo_c2_1}
        \begin{cases}
            p^u(kp+c)\equiv p^v(hp+d)\pmod{p^{n-u}} \\
            p^u(kp+c)\equiv p^v(hp+d)\pmod{p^{n-v}} 
        \end{cases}
    \end{equation}
    Observe that, since $u\ge\ceil{\frac{n}{2}}$, we have that
    $$p^{n-u} \le p^{n-\ceil{\frac{n}{2}}} \le p^{\frac{n}{2}} \le p^{\ceil{\frac{n}{2}}} \le p^u$$
    then $p^u\equiv 0\pmod{p^{n-u}}$. In the same way, since $v < \ceil{\frac{n}{2}}$ we have that $v \le \floor{\frac{n}{2}}$ and
    $$p^{n-v} \ge p^{n-\floor{\frac{n}{2}}} \ge p^{\frac{n}{2}} \ge p^{\floor{\frac{n}{2}}} \ge p^v$$
    that is $p^{n-v}\ge p^v$, and the equality holds if and only if $v=\frac{n}{2}$ but this is impossible since $v < \ceil{\frac{n}{2}}$. Thus $p^{n-v} > p^v$ and $p^v\not\equiv 0\pmod{p^{n-v}}$. If we neglect the invertible factors, we may rewrite \eqref{eqn:cpn_neo_c2_1} equivalently as
    \begin{equation} \label{eqn:cpn_neo_c2_2}
        \begin{cases}
            p^v\equiv 0\pmod{p^{n-u}} \\
            p^u\not\equiv 0\pmod{p^{n-v}} 
        \end{cases}
    \end{equation}
    We need to distinguish among two cases. If $u+v\ge n$, then $u\ge n-v$ implies that $p^u\equiv 0\pmod{p^{n-v}}$, in contradiction with \eqref{eqn:cpn_neo_c2_2}. Otherwise, if $u+v < n$, then $v < n-u$ implies that $p^v\not\equiv 0\pmod{p^{n-u}}$, in contradiction with \eqref{eqn:cpn_neo_c2_2}. Therefore, \thref{thm:cpn_neo_c} does not hold and then $\gamma_{u}^{k,c}$ and $\gamma_{v}^{h,d}$ do not mutually normalize each other. 
\end{proof}

\begin{proposition} \thlabel{thm:cpn_neo_c3}
    Let $\gamma_{u}^{k,c},\gamma_{v}^{h,d}\in\Gamma$ be two gamma functions such that $$1\le v < u <\ceil{\frac{n}{2}}.$$ Then $\gamma_{u}^{k,c}$ and $\gamma_{v}^{h,d}$ do not mutually normalize each other.
\end{proposition}

\begin{proof}
    Consider only the second equation of \thref{thm:cpn_neo_c}
    \begin{equation} \label{eqn:cpn_neogc_2}
        p^u(kp+c)\equiv p^v(hp+d)\pmod{p^{n-v}}.
    \end{equation}
    or, equivalently,
    $$p^v\left(p^{u-v}(kp+c)-hp-d\right)\equiv 0\pmod{p^{n-v}}$$
    where the term $p^{u-v}(kp+c)-hp-d$ is not divisible by $p$, hence invertible. Therefore the equation \eqref{eqn:cpn_neogc_2} is equivalent to $p^v\equiv 0\pmod{p^{n-v}}$ which is false because $v<\ceil{\frac{n}{2}}$ implies that $p^v < p^{n-v}$. Then \thref{thm:cpn_neo_c} does not hold and the conclusion follows.
\end{proof}

\begin{proposition} \thlabel{thm:cpn_neo_cu}
    Let $\gamma_{u}^{k,c},\gamma_{u}^{h,d}\in\Gamma_u$ be two gamma functions such that $1\le u <\ceil{\frac{n}{2}}$. Then $\gamma_{u}^{k,c}$ and $\gamma_{u}^{h,d}$ mutually normalize each other if and only if 
    \[
    \begin{cases}
        k\equiv h\pmod{p^{n-2u-1}} \\
        c=d
    \end{cases}
    \]
\end{proposition}

\begin{proof}
    This is an easy consequence of \thref{thm:cpn_neo_c} when $u=v$, indeed 
    \[
    \begin{aligned}
        p^u(kp+c)\equiv p^u(hp+d) &\iff p^u(kp+c-hp-d)\equiv 0 \pmod{p^{n-u}}\iff (k-h)p+(c-d)\equiv 0\pmod{p^{n-2u}} \\
        &\iff 
        \begin{cases}
            (k-h)p\equiv 0\pmod{p^{n-2u}} \\
            c-d\equiv 0\pmod{p^{n-2u}} \\
        \end{cases}
        \iff 
        \begin{cases}
            k-h\equiv 0\pmod{p^{n-2u-1}} \\
            c-d=0 \\
        \end{cases}
    \end{aligned}
    \]
    because $0\le k,h<p^{n-1}$ and $1\le c,d<p$, and $(k-h)p$ is a multiple of $p$ but $p\nmid (d-c)$.
\end{proof}

\begin{proposition} \thlabel{thm:cpn_neo_cu1}
    For every fixed $1\le u < \ceil{\frac{n}{2}}$, $0\le t<p^{n-2u-1}$, and $1\le c<p$, the family $$A_u^{t,c}=\left\{\gamma_{u}^{k,c}\in\Gamma\, :\, k\equiv t\pmod{p^{n-2u-1}}\right\}$$
    is composed of $p^u$ gamma functions, and they mutually normalize each other. In total, there are $$\frac{1}{p+1}\left(p^{n-1}-p^{n-2\ceil{\frac{n}{2}}+1}\right)$$ distinct $A_u^{t,c}$.
\end{proposition}

\begin{proof}
    Observe that, once fixed $1\le u < \ceil{\frac{n}{2}}$ and $1\le c<p$, two families $A_{u}^{t_1,c}, A_{u}^{t_2,c}$ have the same number of elements because every $\gamma_{u}^{k,c}$ is defined modulo $p^{n-u}$ and there is a bijection 
    \[
    \begin{aligned}
        \varphi\colon A_{u}^{t_1,c}&\to A_{u}^{t_2,c} \\
        \gamma_{u}^{k,c} &\mapsto\gamma_{u}^{h,c}
    \end{aligned}
    \]
    where $k$ and $h$ are such that $k=q(p^{n-2u-1})+t_1$ and $h=q(p^{n-2u-1})+t_2$, for the same $q\in\Z$. Therefore, recalling that the conjugacy class of $\gamma_{u}^{k,c}$ has $p^{n-u}-p^{n-u-1}$ different elements, dividing by all the possible choices of $t$ and $c$, we obtain that $$|A_u^{t,c}|=\frac{p^{n-u}-p^{n-u-1}}{p^{n-2u-1}\cdot(p-1)}=p^u$$ for every $1\le u < \ceil{\frac{n}{2}}$, $0\le t<p^{n-2u-1}$ and $1\le c<p$. 
    
    Moreover, for every fixed $1\le u < \ceil{\frac{n}{2}}$, $0\le t<p^{n-2u-1}$ and $1\le c<p$, there are $p^{n-2u-1}(p-1)$ distinct $A_u^{t,c}$, therefore in total they are
    $$\sum_{u=1}^{\ceil{\frac{n}{2}}-1}\sum_{t=0}^{p^{n-2u-1}-1}\sum_{c=1}^{p-1}1=\sum_{u=1}^{\ceil{\frac{n}{2}}-1}p^{n-2u-1}(p-1) = \frac{1}{p+1}\left(p^{n-1}-p^{n-2\ceil{\frac{n}{2}}+1}\right)$$
    that is the conclusion.
\end{proof}

Since we have taken into account all the possibilites, we can also conclude that

\begin{proposition} \thlabel{thm:cpn_neo_n1}
   There are no other mutual normalizations among pairs of elements of $\Gamma$.
\end{proposition}

\begin{theorem}[Local normalizing graph of $\C_{p^n}$] \thlabel{thm:cpngraph}
    In the notation of the previous propositions, the local normalizing graph of $\C_{p^n}$ contains the following connected components.
    \begin{enumerate}
        \item The clique $H$, composed by $p^{n-\ceil{\frac{n}{2}}}$ regular subgroups.
        \item $\frac{1}{p+1}\left(p^{n-1}-p^{n-2\ceil{\frac{n}{2}}+1}\right)$ connected components $A_u^{t,c}$, each of which is a clique composed by $p^u$ regular subgroups, for $1\le u < \ceil{\frac{n}{2}}$, $0\le t<p^{n-2u-1}$, and $1\le c<p$.
    \end{enumerate}
\end{theorem}

\begin{remark}
    The form of the local normalizing graph of $\C_{p^n}$ strongly depends on the choice of the prime $p$. Indeed, it is composed only by disjoint cliques, each of which has $p^u$ vertices, for some $u$, as stated above. Therefore, since it is easier to understand if compared with the case p=2, we decided not to display it, in order not to lack of generality. 
\end{remark}

\section{Conclusion}

In this work, we have presented an application of the theory of gamma function in order to classify the mutually normalizing regular subgroups of the holomorph of a cyclic group of prime power order, and we have discovered the algebraic conditions of the structure of such groups that constrain the local normalizing graphs in their highly symmetrical shape. Since cyclic groups are the elementary building blocks with which we can construct every finite abelian group: it is ambitious, albeit natural, to wonder for a solution to the \emph{mutual normalization problem} for all the abelian groups. Despite, heuristically, it seems that a general pattern does not exist, we conclude this paper formulating an open problem which would extend out construction.

\begin{problem}
    Describe and classify the local normalizing graph for all finite abelian groups.
\end{problem}

\begin{acknowledgements}
This paper is based on the research conducted while I was working on the Master's Thesis at the University of Trento (Italy). I would like to express my gratitude to Professor Andrea Caranti, my thesis advisor, for his patient guidance and enthusiastic encouragement. My grateful thanks are also extended to Professor David Stanovský, my actual Ph.D. supervisor, for the review and the useful critiques of this research work.
\end{acknowledgements}

\printbibliography

\end{document}